\documentclass[11pt,reqno]{amsart}
\usepackage{amsthm, amsfonts, amssymb, color}
 \usepackage{mathrsfs}
\usepackage{enumerate}
\usepackage{amsmath}
 \usepackage{amstext, amsxtra}
  \usepackage{txfonts}
 \usepackage[colorlinks, linkcolor=black, citecolor=blue, pagebackref, hypertexnames=false]{hyperref}
\usepackage{graphicx}
\usepackage[symbol]{footmisc}
\usepackage{cite}
 \allowdisplaybreaks
\newtheorem{theorem}{Theorem}[section]
\newtheorem{proposition}[theorem]{Proposition}

\newtheorem{lemma}[theorem]{Lemma}

\newtheorem{remark}[theorem]{Remark}

\newtheorem{examples}[theorem]{Examples}
\newtheorem{claim}[theorem]{Claim}

\numberwithin{equation}{section}
\theoremstyle{definition}

\title[]
 {Unique continuation properties for one dimensional higher order Schr\"{o}dinger equations}

\author{Tianxiao Huang, Shanlin Huang\textsuperscript{*} and Quan Zheng}

\address{Tianxiao Huang, School of Mathematics (Zhuhai), Sun Yat-sen University, Zhuhai, Guangdong 519082, China}
\email{htx5@mail.sysu.edu.cn}

\address{Shanlin Huang (corresponding author), School of Mathematics and Statistics, Hubei Key Laboratory of Engineering Modeling and Scientific Computing, Huazhong University of Science and Technology, Wuhan 430074, Hubei, PR China}
\email{shanlin\_huang@hust.edu.cn}

\address{Quan Zheng, School of Mathematics and Statistics, Huazhong University of Science and Technology, Wuhan 430074, Hubei, PR China }
\email{qzheng@hust.edu.cn}

\begin{document}

\begin{abstract}
We study two types of unique continuation properties for the higher order Schr\"{o}dinger equation with potential
$$
i\partial_tu=(-\Delta_x)^mu+V(t,x)u,\quad(t,x)\in\mathbb{R}^{1+n},\,2\leq m\in\mathbb{N}_+.
$$
The first one says if $u$ has certain exponential decay at two times, then $u\equiv0$, and this result is sharp by constructing critical non-trivial solutions. The second one says if $u\equiv0$ in an arbitrary half-space of $\mathbb{R}^{1+n}$, then $u\equiv0$ identically. The uniqueness theorems are given when $n=1$, but we also prove partial results when $n\in\mathbb{N}_+$ for their own interests. Possibility or obstacles to proving these unique continuation properties in higher spatial dimensions are also discussed.
\end{abstract}

\maketitle



\section{Introduction}

Consider the higher order Schr\"{o}dinger equation with potential
\begin{equation}\label{eq0.1}
i\partial_tu=(-\Delta_x)^mu+V(t,x)u,\quad(t,x)\in\mathbb{R}^{1+n},\,2\leq m\in\mathbb{N}_+,
\end{equation}
where $\Delta_x=\partial_{x_1}^2+\cdots+\partial_{x_n}^2$ is the spatial Laplacian. In the case $m=2$, this type of equation was originated in a non-linear setting introduced by Karpman \cite{Karpman}, for considering the effect of higher order dispersion term on the self-focusing. The general linear analysis of \eqref{eq0.1} has been more focused on, for examples, estimates for the fundamental solution and its regularity for the free case \cite{Mi80,Mi81,KPV-91}, $L^p-L^q$ and $L^p$ estimates \cite{ZYF}, wave operators and scattering theory \cite{Hor2,DZF}, and the dispersive estimates \cite{FSY,FSWY} recently.

The aim of the current paper is to inspire more unique continuation properties for equation \eqref{eq0.1}, mainly by studying the Carleman estimates. The two types of problems we consider are much more known in the second order case $m=1$, i.e. the Schr\"{o}dinger equation, for their interesting connections to two important topics in harmonic analysis: the Hardy's uncertainty principle and Fourier restriction estimate. The unique continuation theorems (Theorem \ref{thm1.1} and Theorem \ref{thm1.2}) we obtain are limited in one spatial dimension, but partial results (Examples \ref{ex12}, Proposition \ref{prop2.4} and Lemma \ref{lm2.1}) and discussions in higher dimensions are also given for expected further studies.

\subsection{$L^2$ quantitative unique continuation from two times}\ 

Our first main result concerns time-independent potentials. Consider
\begin{equation}\label{eq1.1}
i\partial_tu=D_x^{2m}u+V(x)u,\quad(t,x)\in\mathbb{R}^2,\,m\geq2,
\end{equation}
where $D_x=i^{-1}\partial_x$.
\begin{theorem}\label{thm1.1}
Suppose $V\in L^\infty(\mathbb{R})$ is real-valued. Then there exists $\tilde{\gamma}>0$ depending only on $m$, such that if $u\in C(\mathbb{R};L^2(\mathbb{R}))$ solves \eqref{eq1.1} and
\begin{equation}\label{eq1.3}
\begin{cases}
e^{\gamma|x|^\frac{2m}{2m-1}}u(T_1,x),\,e^{\gamma|x|^\frac{2m}{2m-1}}u(T_2,x)\in L_x^2(\mathbb{R}),\\
\gamma>\frac{\tilde{\gamma}}{(T_2-T_1)^\frac{1}{2m-1}},
\end{cases}
\end{equation}
for some $T_1<T_2$, then $u\equiv0$.
\end{theorem}

We first remark that no further regularity assumption is needed than $u$ being a continuous solution, i.e. $u\in C(\mathbb{R};L^2(\mathbb{R}))$ is only assumed to satisfy $u(t)=e^{-it(D_x^{2m}+V)}u(0)$. Moreover, Theorem \ref{thm1.1} is sharp in regard to the decay exponent $\frac{2m}{2m-1}$, by the following examples where non-trivial solutions having such exponential decay can be constructed in all spatial dimensions.

\begin{examples}\label{ex12}
For any $n\in\mathbb{N}_+$ and real-valued $V\in L^\infty(\mathbb{R}^n)$, there exists a non-trivial $u\in C^\infty(\mathbb{R};H^{2m}(\mathbb{R}^n))$ solving
\begin{equation}\label{eq14}
i\partial_tu=(-\Delta_x)^mu+V(x)u,\quad(t,x)\in\mathbb{R}^{1+n},
\end{equation}
such that
\begin{equation}\label{1.5}
\left\|e^{h(t)|x|^\frac{2m}{2m-1}}u(t,x)\right\|_{L_x^2(\mathbb{R}^n)}\leq C(1+|t|)^n,\quad t\in\mathbb{R},
\end{equation}
for some function $h$ strictly decreasing in $|t|$, satisfying
\begin{equation}\label{e16}
h(t)\leq\frac{C}{(1+|t|)^\frac{2m}{2m-1}}.
\end{equation}
Moreover if $V\equiv0$, the non-trivial $u$ can be found analytic in $\mathbb{R}^{1+n}$, satisfying
\begin{equation}\label{1.6}
|u(t,x)|\leq C\exp\left\{-\frac{c|x|^\frac{2m}{2m-1}}{(1+|t|)^\frac{2m}{2m-1}}\right\},\quad(t,x)\in\mathbb{R}^{1+n},
\end{equation}
for some $C,\,c>0$. 
\end{examples}

Consequently, if Theorem \ref{thm1.1} holds for \eqref{eq14} concerning decay $e^{-c|x|^{2m/(2m-1)}}$ in higher dimensions, it should also be sharp, but this is unknown even in the case $V\equiv0$. We note that constant $(1+|t|)^{-\frac{2m}{2m-1}}$ in \eqref{e16} and \eqref{1.6} is weaker than the one in \eqref{eq1.3}, (the one in \eqref{eq1.3} can be checked by scaling,) and this is due to the use of higher order heat kernel for the complex time variable in the constructions, see subsection \ref{sb2.4}.

Unique continuation from two times for dispersive equations seems to originate in control theory since the works \cite{Zhang92,Zhang97} of Zhang, and it has attracted plenty of attention in the last three decades. For nonlinear equations, a typical consideration is to assume spatial support constrains for the difference of two solutions at two times, and conclude that difference is identically zero for the time in between. See \cite{KPV-02,KPV-03} for such study on the $k$-generalized KdV equation, and \cite{KPV-03CPAM,IK,IK06} on the nonlinear Schr\"{o}dinger equations. For higher order models, such problem was considered in \cite{DZ} for non-linear one dimensional higher order Schr\"{o}dinger equation, and in \cite{KPV-03'} for higher order KdV type equation \eqref{e110} that we will turn back to later. We also mention that recently, the works \cite{WWZ,HS} for linear Schr\"{o}dinger equations and \cite{LW} for the KdV equation also considered the sizes of supports at two times in the view of observability.

Other than constraining the support, Escauriaza et al. \cite{EKPV-CPDE06} first realized for Schr\"{o}dinger equations, that assuming certain quantitative spatial decay of the solution at two times is enough to ensure uniqueness. Such idea was motivated by an interesting relation between the free Schr\"{o}dinger equation and the Hardy's uncertainty principle in harmonic analysis. The Hardy's uncertainty principle says: \textit{if $f(x)=O(e^{-\beta|x|^2})$, $\hat{f}(\xi)=O(e^{-\alpha|\xi|^2})$ and $\alpha\beta>\frac14$, then $f\equiv0$; if $\alpha\beta=\frac14$, then $f$ is a constant multiple of $e^{-\beta|x|^2}$.} Here
\begin{equation*}
\hat{f}(\xi)=(\mathscr{F}f)(\xi)=\int_{\mathbb{R}^n}e^{-ix\cdot\xi}f(x)dx.
\end{equation*}
In \cite{EKPV-CPDE06}, this principle was pointed out in fact equivalent to a unique continuation property for the free Schr\"{o}dinger equation
\begin{equation}\label{1.8}
i\partial_tu=-\Delta_xu,\quad(t,x)\in\mathbb{R}^{1+n},
\end{equation}
which says: \textit{if $u(0,x)=O(e^{-\beta|x|^2})$, $u(T,x)=O(e^{-\alpha|x|^2})$ and $\alpha\beta>\frac{1}{16T^2}$, then $u\equiv0$; if $\alpha\beta=\frac{1}{16T^2}$, then $u(0,x)$ is a constant times $e^{-(\beta+\frac{i}{4T})|x|^2}$.} Such connection is given by the Fourier expression of the solution
\begin{equation}\label{eq1.4.1}
u(t,x)=(4\pi it)^{-\frac n2}e^\frac{i|x|^2}{4t}\mathscr{F}\left(e^\frac{i|\cdot|^2}{4t}u(0,\cdot)\right)\left(\frac{x}{2t}\right),
\end{equation}
which is a computational result by the fundamental solution $(4\pi it)^{-\frac n2}e^\frac{i|x-y|^2}{4t}$ of \eqref{1.8}.

In a series of works \cite{EKPV-CPDE06,EKPV-JEMS08,EKPV-JLMS11,EKPV-DUKE10,EKPV-CMP11}, Escauriaza et al. extended such type of results to the case of variable coefficients, which applied to the difference of two solutions of certain class of non-linear Schr\"{o}dinger equations. See \cite{EKPV-BAMS12} for a more complete overview on this topic. It should be noticed that the formula \eqref{eq1.4.1} is hardly useful in the case of variable coefficients, and Escauriaza et al. have developed an \textit{equation-based scheme} to prove quantitative uniqueness for Schr\"{o}dinger equations, which combines upper bound estimates in appropriate weighted energy spaces, and Carleman estimates which match the upper bounds quantitatively.

It is then natural to ask whether such scheme can give quantitative unique continuation properties from two times for more general dispersive equations, while to find perfect matching uncertainty principles seems rarely possible. In \cite{EKPV-JFA07}, Escauriaza et al. considered the $k$-generalized KdV equations
\begin{equation*}
\partial_tu+\partial_x^3u+u^k\partial_xu=0,\quad(t,x)\in[0,1]\times\mathbb{R},\,k\in\mathbb{N}_+,
\end{equation*}
and obtained that if the difference of two solutions $\omega=u_1-u_2$ and its derivatives satisfy spatial decay at two times
\begin{equation*}
\omega(0,x),\,\omega(1,x)\in H^1(e^{ax_+^{3/2}}dx),\quad a=a(u_1,u_2,k),
\end{equation*}
then $u_1\equiv u_2$. Here $x_+=\max\{0,x\}$, and it was pointed out that the decay $e^{-cx_+^{3/2}}$, which is the same spatial decay of the fundamental solution for the linear problem appearing as a scaled Airy function, is optimal by constructing a non-trivial solution. We refer to the recent works \cite{BJM,CFL} for Zakharov-Kuznetsov equations as higher dimensional analogues.

Such problem for higher order dispersive equations seems computationally difficult in their own nature. An odd order model closely related to the higher order Schr\"{o}dinger equations \eqref{eq0.1} or \eqref{eq1.1} is the higher order KdV type equation
\begin{equation}\label{e110}
\partial_tu+(-1)^{k+1}\partial_x^{2k+1}u+P(u,\partial_xu,\dots,\partial_x^pu)=0,\quad(t,x)\in\mathbb{R}^2,\,k\geq2,
\end{equation}
where $P$ is a polynomial. In \cite{Is}, Isaza proved that (see Dawson \cite{Da} for an earlier study for $k=2$,) if the difference of two solutions $\omega=u_1-u_2$ satisfies
\begin{equation*}\label{e1.11}
\omega(0,x),\,\omega(1,x)\in
\begin{cases}
L^2(e^{x_+^{4/3+\epsilon}}dx),&\text{if $p\leq2k-1$ and $\epsilon>0$},\\
L^2(e^{ax_+^{(2k+1)/2k}}dx),\quad&\text{if $p\leq k$ and $a=a(k)$},
\end{cases}
\end{equation*}
then $u_1\equiv u_2$. We remark that in the less non-linear case $p\leq k$, the decay $e^{-cx_+^{(2k+1)/2k}}$ was pointed out to be the same spatial decay of the fundamental solution for the linear problem.

To motivate a similar form of quantitative uniqueness for the higher order Schr\"{o}dinger equation \eqref{eq0.1}, we in the first place however, lack of a formula like \eqref{eq1.4.1} to make a connection between the solution of the free case
\begin{equation}\label{e1.12}
\begin{cases}
i\partial_tu=(-\Delta_x)^mu,\quad(t,x)\in\mathbb{R}^{1+n},\,m\geq2,\\
u(0,x)=u_0,
\end{cases}
\end{equation}
and the Fourier transform of the initial data $u_0$, thus any uncertainty principle for Fourier transform seems not immediately inspiring. From another perspective of fundamental solution as just mentioned for the KdV and higher order KdV cases, one may also notice that the exponential decay condition in Theorem \ref{thm1.1} is not a reflection of \eqref{e1.12}'s fundamental solution, since which only obeys polynomial decay (e.g. \cite[Examples 3.3]{HHZ})
\begin{equation*}
\left|\mathscr{F}^{-1}(e^{-it|\cdot|^{2m}})(t,x-y)\right|\leq C|t|^{-\frac{n}{2m}}(1+|t|^{-\frac{1}{2m}}|x-y|)^{-\frac{n(m-1)}{2m-1}},\quad t\neq0,\,x,\,y\in\mathbb{R}^n.
\end{equation*}

In fact, we think of Theorem \ref{thm1.1} as a "limit" of its closely related parabolic version: \textit{the quantitative backward unique continuation}. For any $\epsilon>0$, consider Cauchy problem
\begin{equation}\label{e1.13}
\begin{cases}
\partial_tu=-(\epsilon+i)(-\Delta_x)^mu,\quad(t,x)\in[0,+\infty)\times\mathbb{R}^n,\,m\geq2,\\
u(0,x)=u_0.
\end{cases}
\end{equation}
Also recall an uncertainty principle originally due to Morgan and Beurling (see Escauriaza et al. \cite{EKPV-JLMS11}): \textit{if $f\in L^1(\mathbb{R})$ or $f\in L^2(\mathbb{R}^n)$ for $n\geq2$, $1<p<2$, $1/p+1/q=1$, and}
\begin{equation}\label{e1.14}
\left\|e^{\frac{\alpha^p|x|^p}{p}}f(x)\right\|_{L^1}+\left\|e^{\frac{\beta^q|\xi|^q}{q}}\hat{f}(\xi)\right\|_{L^1}<\infty,\quad\alpha\beta\geq1,
\end{equation}
\textit{then $f\equiv0$.} Now since the solution $u$ to \eqref{e1.13} has Fourier transform
\begin{equation}\label{1.17}
\hat{u}(t,\xi)=e^{-(\epsilon+i)t|\xi|^{2m}}\hat{u_0}(\xi),
\end{equation}
if we put $f(x)=u(T,x)$, $p=\frac{2m}{2m-1}$ and $q=2m$ in \eqref{e1.14}, the following is immediate: \textit{if $u_0\in L^1(\mathbb{R})$ or $u_0\in L^2(\mathbb{R}^n)$ for $n\geq2$, and}
\begin{equation*}
e^{c|x|^\frac{2m}{2m-1}}u(T,x)\in L_x^1,\quad c>\frac{2m-1}{(2m)^\frac{2m}{2m-1}(\epsilon T)^\frac{1}{2m-1}},
\end{equation*}
\textit{then $u_0\equiv0$.}

The above argument fails for the equation \eqref{e1.12}, because $\hat{u}(t,\xi)=e^{-it|\xi|^{2m}}\hat{u_0}(\xi)$ does not have a priori decay like \eqref{1.17}. However, if in addition the initial data $u_0$ decays like $e^{-c|x|^{2m/(2m-1)}}$, the situation may be compensated. Such possibility is encouraged by the following property of Fourier transform. Let $p\geq1$ and $Z_p^p$ be the class of analytic functions $\varphi$ in $\mathbb{C}^n$ satisfying 
\begin{equation*}
|\varphi(z)|\leq Ce^{\sum_{j=1}^{n}\epsilon_jC_j|z_j|^p},\quad z\in\mathbb{C}^n,
\end{equation*}
where $C,\,C_j>0$, $\epsilon_j=-1$ when $z_j\in\mathbb{R}$, and $\epsilon_j=1$ when $z_j\notin\mathbb{R}$. In Gel'fand and \v{S}ilov \cite{GS} it was proved that $\mathscr{F}Z_p^p=Z_q^q$ where $1/p+1/q=1$. Now consider \eqref{e1.12} with $u_0\in Z_{2m/(2m-1)}^{2m/(2m-1)}$, we have $\hat{u_0}\in Z_{2m}^{2m}$ and in particular,
\begin{equation*}\label{A}
|\hat{u}(t,\xi)|=|e^{-it|\xi|^{2m}}\hat{u_0}(\xi)|=|\hat{u_0}(\xi)|\leq C'e^{-C''|\xi|^{2m}}.
\end{equation*}

This, and the backward uniqueness argument suggest that it is reasonable to work with the decay $e^{-c|x|^{2m/(2m-1)}}$ given at two times. In fact, such decay persists for all times in between.

\begin{proposition}\label{prop2.4}
Suppose $n\in\mathbb{N}_+$, $V\in L^\infty(\mathbb{R}^n)$ is real-valued, and $u\in C(\mathbb{R};L^2(\mathbb{R}^n))$ solves \eqref{eq14}. If there exists $\gamma>0$ such that
\begin{equation*}
e^{\gamma|x|^\frac{2m}{2m-1}}u(0,x),\,e^{\gamma|x|^\frac{2m}{2m-1}}u(T,x)\in L_x^2(\mathbb{R}^n),
\end{equation*}
for some $T>0$, then for $t\in[0,T]$, we have
	\begin{equation}\label{eq2.15}
	\left\|e^{\gamma|x|^\frac{2m}{2m-1}}u(t,x)\right\|_{L_x^2}\leq Ce^{\frac{t(T-t)}{4}||V||^2_{L^\infty}}\left\|e^{\gamma|x|^\frac{2m}{2m-1}}u(0,x)\right\|_{L_x^2}^\frac{T-t}{T}\left\|e^{\gamma|x|^\frac{2m}{2m-1}}u(T,x)\right\|_{L_x^2}^\frac{t}{T}.
	\end{equation}
\end{proposition}

Considering logarithmic convexity of weighted energy was initiated in Escauriaza et al. \cite{EKPV-JEMS08} for Schr\"{o}dinger equations with Gaussian weight, and variants were obtained in \cite{EKPV-MRL08}. The idea there was to bootstrap the formal calculation for the second order time derivative of, for example, $\log\|e^{a(t)|x|^2}u(t,x)\|_{L_x^2}^2$ with some well chosen $a(t)$. What is somehow different in proving Proposition \ref{prop2.4} is that, instead of working with the weight $e^{\gamma|x|^{2m/(2m-1)}}$ directly, we shall first consider proving with the linear exponential weight $e^{\lambda\cdot x}$ for all $\lambda\in\mathbb{R}^n$, and then lift the estimate by a subordination inequality proved in \cite{EKPV-JLMS11}
\begin{equation}\label{eq2.27}
c_{n,p}^{-1}e^{|x|^p/p}\leq\int_{\mathbb{R}^n}e^{\lambda\cdot x-|\lambda|^q/q}|\lambda|^\frac{n(q-2)}{2}d\lambda\leq c_{n,p}e^{|x|^p/p},\quad x\in\mathbb{R}^n,
\end{equation}
where $p\in(1,2]$ and $1/p+1/q=1$. The point of such approach is to formally keep the conjugated operator $e^{\lambda\cdot x}(-\Delta_x)^me^{-\lambda\cdot x}$ having constant coefficients, so that the energy method is more applicable to $\log\|e^{\lambda\cdot x}u(t,x)\|_{L_x^2}^2$. This formal argument seems irrelevant to the exponent $\frac{2m}{2m-1}$, however, as indicated in Examples \ref{ex12}, the bootstrapping requires knowledge from the higher order analytic semigroup $\{e^{-z((-\Delta)^m+V)}\}_{\mathrm{Re}\,z>0}$, the kernel of which has spatial decay $e^{-c|x|^{2m/(2m-1)}}$. We remark that the solution $u$ is neither assumed to be more regular than continuous in $L_x^2$, nor in advance to have any spatial decay when $t\in(0,T)$, but working with the regularized solution $e^{-(\epsilon+it)((-\Delta)^m+V)}u(0)$ will gain both.

Following the upper/lower-scheme for proving uniqueness, our final step needs a quantitative Carleman estimate that matches Proposition \ref{prop2.4} regarding the exponent $\frac{2m}{2m-1}$. But we are only able to do so in spatial dimension one.

\begin{proposition}\label{lm2.6}
	Suppose $\phi\in C^2([0,1])$ is real-valued, and define
	\begin{equation}\label{e118}
	\mbox{$Q(t,x)=2\gamma R^\frac{2m}{2m-1}\left(\frac xR+\phi(t)\right)^2,\quad\gamma,\,R>0,~(t,x)\in[0,1]\times\mathbb{R}.$}
	\end{equation}
	 Consider all $u\in C_c^\infty((0,1)\times\mathbb{R})$ such that $0<d_1\leq|\frac xR+\phi(t)|\leq d_2$ holds in $\mathrm{supp}\,u$ for some fixed $d_1<d_2$. Then there exist $\gamma_0,R_0,C>0$ such that
	\begin{equation}\label{eq2.30}
	\iint e^Q\left|D_tu+D_x^{2m}u\right|^2dxdt\geq C\gamma^{4m-1}R^\frac{2m}{2m-1}\iint e^Q|u|^2dxdt,\quad\gamma\geq\gamma_0,\,R\geq R_0.
	\end{equation}
	Here $D_t=i^{-1}\partial_t$.
\end{proposition}

In the proof of Theorem \ref{thm1.1}, exponent $\frac{2m}{2m-1}$ in \eqref{e118} plays the role of matching the weight $e^{\gamma|x|^{2m/(2m-1)}}$ in Proposition \ref{prop2.4}. We mention without proof that when $n\geq2$, if $\tilde{Q}(t,x)=2\gamma R^\frac{4m}{3m-1}|\frac xR+\phi(t)e_1|^2$ where $e_1$ the first unit vector in $\mathbb{R}^n$, a similar techniques shows the lower bound
\begin{equation*}
\int_{[0,1]\times\mathbb{R}}e^{\tilde{Q}}\left|D_tu+(-\Delta_x)^mu\right|^2dxdt\geq C\gamma^{3m}R^\frac{4m}{3m-1}\int_{[0,1]\times\mathbb{R}}e^{\tilde{Q}}|u|^2dxdt.
\end{equation*}
If a unique continuation result for \eqref{eq14} is expected with the exponent $\frac{4m}{3m-1}$ involved, since $\frac{4m}{3m-1}>\frac{2m}{2m-1}$ when $m>1$, we more or less need an analogue of Proposition \ref{prop2.4} on the faster decay $e^{-c|x|^{4m/(3m-1)}}$, which is however not evident.

\subsection{$L^p$ global unique continuation}\ 

Our second main result concerns unbounded potentials over the whole space-time. Consider
\begin{equation}\label{e121}
i\partial_tu=D_x^{2m}u+V(t,x)u,\quad(t,x)\in\mathbb{R}^2,\,m\geq2.
\end{equation}

\begin{theorem}\label{thm1.2}
Suppose $D$ is an arbitrary half-plane in $\mathbb{R}^2$, $V\in L^\frac{2m+1}{2m}(\mathbb{R}^2)$, and $u\in W^{\frac{4m+2}{4m+1}}(\mathbb{R}^2)$ where
\begin{equation*}
W^\frac{4m+2}{4m+1}(\mathbb{R}^2)=\left\{u\in L^\frac{4m+2}{4m+1 }(\mathbb{R}^2);~\partial_tu,\,D_x^{2m}u\in L^\frac{4m+2}{4m+1}(\mathbb{R}^2)\right\}.
\end{equation*}
If $u$ is a solution to \eqref{e121} and $u\equiv0$ in $D$, then $u\equiv0$ in $\mathbb{R}^2$.
\end{theorem}

\begin{remark}
If $V$ is analytic and $\partial D$ is non-characteristic to $i\partial_t-D_x^{2m}$, the conclusion is trivial by the well-known Holmgren's uniqueness theorem. If $V\equiv0$, and $\partial D=\{(t,x);~t=0,\,x\in\mathbb{R}\}$ which is characteristic, since one is able to construct by H\"{o}rmander \cite[Theorem 8.6.7]{Hor1} a non-trivial $C^{\infty}$ solution $u$ with $\mathrm{supp}\,u=D$, our result implies $u\notin W^{\frac{4m+2}{4m+1}}(\mathbb{R}^2)$. Thus for a global unique continuation property like Theorem \ref{thm1.2} to hold, certain regularity condition must be needed.
\end{remark}

Theorem \ref{thm1.2} was first proved in Kenig and Sogge \cite{KS} for the Schr\"{o}dinger equation
\begin{equation}\label{e122}
i\partial_tu=-\Delta_xu+V(t,x)u,\quad(t,x)\in\mathbb{R}^{1+n},\,n\in\mathbb{N}_+,
\end{equation}
where $V\in L^\frac{n+2}{2}(\mathbb{R}^{1+n})$ and $D$ is an arbitrary half-space in $\mathbb{R}^{1+n}$. Such unique continuation property was shown by establishing the $L^p$ Carleman estimate
\begin{equation}\label{e123}
\left\|e^{\lambda\langle(t,x),v\rangle}u\right\|_{L^\frac{2n+4}{n}(\mathbb{R}^{1+n})}\leq C\left\|e^{\lambda\langle(t,x), v\rangle}(i\partial_t+\Delta_x)u\right\|_{L^\frac{2n+4}{n+4}(\mathbb{R}^{1+n})},\quad u\in C_c^\infty(\mathbb{R}^{1+n}),
\end{equation}
where $C>0$ is independent of $(\lambda,v)\in\mathbb{R}\times\mathbb{S}^{n}$, and $\langle\cdot,\cdot\rangle$ is the inner product in $\mathbb{R}^{1+n}$. The pattern for proving uniqueness by \eqref{e123} is now quite known, and it is no surprise that an $L^p$ Carleman estimate like \eqref{e123} allows one to consider unbounded potentials with certain integrability in the view of H\"{o}lder's inequality. On extending results for equation \eqref{e122} or the associated Carleman estimate \eqref{e123}, Lee and Seo \cite{LS} as well as Seo \cite{Seo12} considered mixed norms in $L_t^sL_x^p$, and it was further extended in Seo \cite{Seo11} to Wiener amalgam norms. Clearly estimates like \eqref{e123} can not be applied to the equation with a time-independent potential $V(x)$, however in the works \cite{Seo14,Seo}, Seo proved certain $L^2$ type Carleman estimates where $|V(x)|$ plays a role of weight in the norms, and global unique continuation results were still obtained. 

We also remark that the special weight $e^{\lambda x_1}$ in the Carleman estimate was considered in many works to show unique continuation from two times when spatial supports are constrained, as previously mentioned in Kenig et al. \cite{KPV-02,KPV-03,KPV-03CPAM} for the KdV and Schr\"{o}dinger equations, \cite{KPV-03'} for higher order KdV type equation \eqref{e110}, and  Duan and Zheng \cite{DZ} for non-linear one dimensional Schr\"{o}dinger equation. In particular, the following Carleman estimate was proved in \cite[Proposition 3.1]{DZ}
\begin{equation}\label{1.24}
\left\|e^{\lambda x}u\right\|_{L^{4m+2}(\mathbb{R}^{2})}\leq C\left\|e^{\lambda x}(i\partial_t-D_x^{2m})u\right\|_{L^\frac{4m+2}{4m+1}(\mathbb{R}^{2})},\quad u\in C_c^\infty(\mathbb{R}^{2}),
\end{equation}
where $C>0$ is independent of $\lambda\in\mathbb{R}$, and we will see later that this is a special case of Lemma \ref{lem3.1} below.

For higher order Schr\"{o}dinger equations like \eqref{eq0.1}, the general topic of \textit{unique continuation across hypersurface} is more understood in the case of bounded coefficients, but actual relevant results are quite few, and the higher dimensional progress is indeed recent. An earlier effort in this direction goes back to Isakov \cite{Isakov}, a conclusion there in one spatial dimension is the local uniqueness (and stability) of the Cauchy problem for 
\begin{equation*}
ia(t,x)\partial_tu+D_x^ku+\sum_{j=1}^{k}b_j(t,x)D_x^ju=f,\quad(t,x)\in(0,T)\times(0,1),\,k\geq2,
\end{equation*}
where the data is given on the spatial boundary $(0,T)\times\{0\}$. Note that this result also applies to higher order KdV type equation \eqref{e110} earlier studied by Saut and Scheurer \cite{SS}. In higher dimensions, the first author recently considered the higher order Schr\"{o}dinger inequality in \cite{H}
\begin{equation}\label{e124}
\left|i\partial_tu-(-\Delta_x)^mu\right|\leq C\begin{cases}
\sum_{|\alpha|\leq[\frac{3m}{2}]}|\partial_x^\alpha u|,&m\geq3,\\
\sum_{|\alpha|\leq2}|\partial_x^\alpha u|+\sum_{|\alpha|=2}|\partial_x^{\alpha+e_1}u|,\quad&m=2,
\end{cases}
\end{equation}
and proved unique continuation across the hyperplane $\{(t,x)\in\mathbb{R}^{1+n};~x_1=0,\,|t|<1\}$. Such property is only global in the $x'=(x_2,\dots,x_n)$ variables, and a completely local result for \eqref{e124} was also proved in \cite{H} concerning saddle-shape hypersurfaces. The just mentioned results of \cite{Isakov,H} were all based on $L^2$ Carleman estimates. 

To consider unbounded potentials and motivate an $L^p$ Carleman estimate for the higher order Schr\"{o}dinger equation \eqref{eq0.1}, we now turn back to the second order case for a review in more details. First note that in Kenig and Sogge \cite{KS}, the idea of proving an $L^p$ global unique continuation property for the Schr\"{o}dinger equation had an implication from the Fourier restriction theory. To be more specific, the special case $\lambda=0$ of $L^p$ Carleman estimate \eqref{e123} is actually a consequence of the Fourier restriction estimate originally due to Strichartz \cite{Stri}
\begin{equation*}
\left\|\int_{\mathbb{R}^n}e^{i\langle(t,x),(|\xi|^2,\xi)\rangle}\hat{f}(|\xi|^2,\xi)d\xi\right\|_{L^\frac{2n+4}{n}(\mathbb{R}^{1+n})}\leq C\|f\|_{L^\frac{2n+4}{n+4}(\mathbb{R}^{1+n})}.
\end{equation*}
Here the Fourier transform is taken in all variables. The proof of \eqref{e123} relies on showing the "uniform Strichartz estimate"
\begin{equation}\label{e125}
\|g\|_{L^\frac{2n+4}{n}(\mathbb{R}^{1+n})}\leq C\|(i\partial_t+\Delta_x+L(D_x))g\|_{L^\frac{2n+4}{n+4}(\mathbb{R}^{1+n})},\quad g\in C_c^\infty(\mathbb{R}^{1+n}),
\end{equation}
where $C>0$ is uniform for all first order differential operators $L(D_x)$ with constant coefficients. This is of course a stronger claim than Carleman estimate \eqref{e123}, for we can write $e^{\lambda\langle(t,x),v\rangle}(i\partial_t+\Delta_x)e^{-\lambda\langle(t,x),v\rangle}=i\partial_t+\Delta_x+L_{\lambda,v}(D_x)$. To show \eqref{e125}, a special case
\begin{equation}\label{e126}
\|g\|_{L^\frac{2n+4}{n}(\mathbb{R}^{1+n})}\leq C\|(i\partial_t+\Delta_x+z)g\|_{L^\frac{2n+4}{n+4}(\mathbb{R}^{1+n})},\quad g\in C_c^\infty(\mathbb{R}^{1+n}),\,z\in\mathbb{C},
\end{equation}
namely the "uniform resolvent estimate" is crucial, because a frequency localization argument in the direction of $1$-homogeneous part of $L(D_x)$ reduces \eqref{e125} to \eqref{e126}. The above heuristics was earlier considered in H\"{o}rmander \cite{Hor83} for elliptic equations and in Kenig et al. \cite{KRS} for hyperbolic equations.

If we follow the same heuristics for the higher order Schr\"{o}dinger equation \eqref{eq0.1}
\begin{equation*}
i\partial_tu=(-\Delta_x)^mu+V(t,x)u,\quad(t,x)\in\mathbb{R}^{1+n},
\end{equation*}
since the relevant restriction property is a known consequence of oscillatory integral estimate (see Stein \cite[p. 369]{St}):
\begin{equation}\label{e127}
\left\|\int_{\mathbb{R}^n}e^{i\langle(t,x),(|\xi|^{2m},\xi)\rangle}\hat{f}(|\xi|^{2m},\xi)d\xi\right\|_{L^\frac{4m+2n}{n}(\mathbb{R}^{1+n})}\leq C\|f\|_{L^\frac{4m+2n}{4m+n}(\mathbb{R}^{1+n})},
\end{equation}
a corresponding Carleman estimate for $i\partial_t-(-\Delta_x)^m$ to consider is on the horizon. But so far we are only able to prove the following in one spatial dimension.

\begin{lemma}\label{lem3.1}
It follows that
\begin{align}\label{equ3.1}
\left\|e^{\lambda\langle(t,x),v\rangle}u\right\|_{L^{4m+2}(\mathbb{R}^{2})}\leq C\left\|e^{\lambda\langle(t,x), v\rangle}(i\partial_t-D_x^{2m})u\right\|_{L^\frac{4m+2}{4m+1}(\mathbb{R}^{2})},\quad u\in C_c^\infty(\mathbb{R}^{2}),
\end{align}
where $C>0$ is independent of $(\lambda,v)\in\mathbb{R}\times\mathbb{S}^{1}$.
\end{lemma}

The special case $\lambda=0$ is of course a result of restriction estimate \eqref{e127} when $n=1$, and the mentioned Carleman estimate \eqref{1.24} is the special case $v=(0,1)$. 

To give a quick overview of what comes different in the higher order case, we shortly sketch the proof of Lemma \ref{lem3.1}. Denoted by $v=(v_0,v_1)$, the scaling $((\lambda v_1)^{2m}t,\lambda v_1x)\rightarrow(t,x)$ reduces \eqref{equ3.1} to the case $\lambda=1$ and $v=(b,1)$ for all $b\in\mathbb{R}$. We thus further set $g=e^{\langle(t,x),(b,1)\rangle}u$ and consider the conjugated form of \eqref{equ3.1}
\begin{equation*}
\left\|g\right\|_{L^{4m+2}(\mathbb{R}^{2})}\leq C\|(i\partial_t-(D_x+i)^{2m}-ib)g\|_{L^\frac{4m+2}{4m+1}(\mathbb{R}^{2})},\quad g\in C_c^\infty(\mathbb{R}^2),\,b\in\mathbb{R}.
\end{equation*}
Apparently, this should be implied by the uniform Fourier multiplier estimate
\begin{equation}\label{eq39}
\left\|\mathscr{F}^{-1}(M_b(\tau, \xi)\hat{f}(\tau,\xi))\right\|_{L^{4m+2}(\mathbb{R}^{2})}\leq C\|f\|_{L^\frac{4m+2}{4m+1}(\mathbb{R}^{2})},
\end{equation}
for all $f=(i\partial_t-(D_x+i)^{2m}-ib)g$ with $g\in C_c^\infty(\mathbb{R}^2)$, where $C>0$ is independent of $b\in\mathbb{R}$, and
\begin{equation}\label{e130}
M_b(\tau,\xi)=\frac{1}{\tau+(\xi+i)^{2m}+ib}.
\end{equation}

Following the scheme introduced in \cite{KS}, our priori to consider is a "uniform resolvent estimate" like \eqref{e126}, and we claim it in a different form for technical reasons.
\begin{claim}\label{cl1.7}
Let $P(\xi)=\mathrm{Re}\,(\xi+i)^{2m}$. Then
\begin{equation}\label{eq310}
\left\|\mathscr{F}^{-1}\left(\frac{\hat{f}(\tau,\xi)}{\tau+P(\xi)+z}\right)\right\|_{L^{4m+2}(\mathbb{R}^{2})}\leq C\|f\|_{L^\frac{4m+2}{4m+1}(\mathbb{R}^{2})},\quad f\in C_c^\infty(\mathbb{R}^2),\,z\in\mathbb{C}\setminus\mathbb{R}.
\end{equation}
\end{claim} 
The reduction from \eqref{eq39} to \eqref{eq310} will be a frequency localization argument by the Littlewood-Paley theory, along with a delicate analysis of the zeros of polynomial $Q_b(\xi)=\mathrm{Im}\,(\xi+i)^{2m}+b$ which is the other part of $M_b(\tau,\xi)$, and the boundedness for a bunch of Fourier multipliers frozen at different frequency scales accordingly to the zeros. When $m>1$, since the zeros, and consequently the Fourier multipliers encountered are all \textbf{implicitly} depending on $b$, the difficulty there will be showing that the estimates for them are uniform in $b\in\mathbb{R}$, for the goal of achieving the uniform estimate \eqref{eq39}. We note that the dependence on $b$ is explicit when $m=1$, and in \cite{KS} the above mechanism in all dimensions was processed in a quite straightforward manner.

Our main obstacle to proving a Carleman estimate in higher spatial dimensions is relevant to a proper version of Claim \ref{cl1.7} that extends the restriction estimate \eqref{e127}. The issue says due to a highly possible phase degeneracy problem, certain oscillatory integral estimate we need seems false in higher dimensions, whereas it is true in dimension one by the van der Corput lemma, and it is also true in the second order case for all dimensions for there is no degeneracy at all. We will discuss this in more details in Remark \ref{rk3.1}. Upon any indefinite resolution of such an issue, however, the higher dimensional argument for the frequency localization seems not clearly similar to that in the proof of Lemma \ref{lem3.1}.

\subsection{Organization and notations}\ 

The rest of this paper is organized as what follows. In section \ref{sec2}, we shall first prove Proposition \ref{prop2.4}, Proposition \ref{lm2.6}, and then combine them to show the first main result Theorem \ref{thm1.1}. The constructions for Examples \ref{ex12} will be placed at the end of this section. In section \ref{sec3}, we shall first show how to use Lemma \ref{lem3.1} to prove our second main result Theorem \ref{thm1.2}. After that, we establish Claim \ref{cl1.7}, and then use it to complete the proof of Lemma \ref{lem3.1}.

Throughout the paper, $C$ denotes a generic positive constant whose value may vary from line to line. We generally use $D=i^{-1}\partial$ for the convenience of taking Fourier transform, and the subscript will be specified when necessary. For $f,\,g\in L^2$, $(f,g)=\int f\bar{g}dx$ denotes the inner product. The Fourier transform of $f$ in $\mathbb{R}^n$ is defined to be $(\mathscr{F}f)(\xi)=\hat{f}(\xi)=\int e^{-ix\cdot\xi}f(x)dx$. We sometimes use $\tilde{f}(t,\xi)=\int e^{-ix\cdot\xi}f(t,x)dx$ to denote the spatial Fourier transform of $f$.


\section{$L^2$ quantitative unique continuation from two times}\label{sec2}

\subsection{Logarithmic convexity of weighted energy}\label{sb2.1}\ 

We shall prove Proposition \ref{prop2.4} in this part. Assume $m,\,n\in\mathbb{N}_+$. First, we need a parabolic decay estimate, the version of which was earlier considered in Escauriaza et al. \cite{EKPV-JEMS08}. Let $K(t,x,y)$ be the kernel of higher order heat semigroup $\{e^{-t((-\Delta_x)^m+V)}\}_{t\geq0}$ on $L^2(\mathbb{R}^n)$ where $V\in L^\infty(\mathbb{R}^n)$ is real-valued. It is known (see Huang et al. \cite[Theorem 1.2]{HWZD}) that there exist $C_1,\,C_2,\,\omega_0>0$ for the following to hold
\begin{equation*}\label{eq2.1}
|K(t,x,y)|\leq C_1t^{-\frac{n}{2m}}\exp\left\{-C_2t^{-\frac{1}{2m-1}}|x-y|^\frac{2m}{2m-1}+\omega_0t\|V\|_{\infty}\right\},\quad t>0,\,x,\,y\in\mathbb{R}^n.
\end{equation*}
By Zheng and Zhang \cite[Theorem 2.1]{ZZ}, it further follows that for some $C,\,c>0$, the kernel $K(z,x,y)$ of the analytic semigroup $\{e^{-z((-\Delta_x)^m+V)}\}_{\mathrm{Re}\,z>0}$ satisfies
\begin{equation}\label{eq2.2}
|K(z,x,y)|\leq C(\mathrm{Re}z)^{-\frac{n}{2m}}\exp\left\{-c\mathrm{Re}\,z\left(\mbox{$\frac{|x-y|}{|z|}$}\right)^\frac{2m}{2m-1}+\omega_0\mathrm{Re}\,z\|V\|_{\infty}\right\},\quad\mathrm{Re}\,z>0,\,x,\,y\in\mathbb{R}^n.
\end{equation}

\begin{lemma}\label{lm2.1}
Suppose $A>0$, $B\in\mathbb{R}$ and $V\in L^\infty(\mathbb{R}^n)$ is real-valued. Then there exist $N_1,\,N_2>0$ which are independent of $A,\,B$ and $V$, such that
\begin{equation}\label{eq2.3}
\left\|e^{\Theta_{A,\,B}(\gamma)|x|^\frac{2m}{2m-1}}e^{-(A+iB)((-\Delta_x)^m+V)}f\right\|_{L^2}\leq N_1e^{\omega_0A\|V\|_{\infty}}\left(1+A^{-2}B^2\right)^\frac n2\left\|e^{\gamma|x|^\frac{2m}{2m-1}}f\right\|_{L^2},
\end{equation}
holds for all $\gamma>0$ and $f$ with $e^{\gamma|x|^{2m/(2m-1)}}f\in L^2(\mathbb{R}^n)$, where $\omega_0$ is the same in \eqref{eq2.2}, and
\begin{equation}\label{eq2.4}
\Theta_{A,\,B}(\gamma)=\frac{\gamma}{(1+N_2A(1+A^{-2}B^2)^m\gamma^{2m-1})^\frac{1}{2m-1}}.
\end{equation}
\end{lemma}

\begin{proof}
Let $z=A+iB$, then \eqref{eq2.2} reads
\begin{equation*}
|K(A+iB,x,y)|\leq\frac{Ce^{\omega_0A\|V\|_{\infty} }}{A^\frac{n}{2m}}\exp\left\{-\frac{cA|x-y|^\frac{2m}{2m-1}}{(A^2+B^2)^\frac{m}{2m-1}}\right\},\quad x,\,y\in\mathbb{R}^n.
\end{equation*}
Denoted by $p=\frac{2m}{2m-1}$, we thus have for $a>0$ that
\begin{equation*}
\begin{split}
\left\|e^{a|x|^p}e^{-(A+iB)((-\Delta)^m+V)}f\right\|_{L^2}\leq&\left\|\int_{\mathbb{R}^n} e^{a|x|^p}|K(A+iB,x,y)||f(y)|dy\right\|_{L^2}\\
\leq&\frac{Ce^{\omega_0A\|V\|_{\infty}}}{A^\frac{n}{2m}}\left\|\int_{\mathbb{R}^n} e^{a|x|^p-b|x-y|^p-\gamma|y|^p}|e^{\gamma|y|^p}f(y)|dy\right\|_{L^2},
\end{split}
\end{equation*}
where $b=cA/(A^2+B^2)^\frac{m}{2m-1}$. If we take
\begin{equation*}
a=\frac{\gamma}{\left(1+\left(\frac{2\gamma}{b}\right)^\frac{1}{p-1}\right)^{p-1}},
\end{equation*}
by the convexity of $|\cdot|^p$, we have for all $x,\,y\in\mathbb{R}^n$ that
\begin{equation*}
\begin{split}
&a|x|^p-b|x-y|^p-\gamma|y|^p\\
=&\gamma\left|\frac{(2\gamma/b)^\frac{1}{p-1}}{1+(2\gamma/b)^\frac{1}{p-1}}\cdot\frac{(1+(2\gamma/b)^\frac{1}{p-1})^\frac1p(x-y)}{(2\gamma/b)^\frac{1}{p-1}}+\frac{(1+(2\gamma/b)^\frac{1}{p-1})^\frac1py}{1+(2\gamma/b)^\frac{1}{p-1}}\right|^p\\
&\quad-\frac{\gamma}{1+(2\gamma/b)^\frac{1}{p-1}}\left|\left(1+(2\gamma/b)^\frac{1}{p-1}\right)^\frac1py\right|^p-b|x-y|^p\\
\leq&\frac{\gamma(2\gamma/b)^\frac{1}{p-1}}{1+(2\gamma/b)^\frac{1}{p-1}}\left|\frac{(1+(2\gamma/b)^\frac{1}{p-1})^\frac1p(x-y)}{(2\gamma/b)^\frac{1}{p-1}}\right|^p-b|x-y|^p\\
=&-\frac b2|x-y|^p.
\end{split}
\end{equation*}
Using Young's inequality we derive
\begin{equation*}
\begin{split}
\left\|\int_{\mathbb{R}^n} e^{a|x|^p-b|x-y|^p-\gamma|x|^p}|e^{\gamma|y|^p}f(y)|dy\right\|_{L^2}\leq&\left\|\int_{\mathbb{R}^n} e^{-\frac12b|x-y|^p}|e^{\gamma|y|^p}f(y)|dy\right\|_{L^2}\\
\leq&C_{m,n}b^{-\frac np}\left\|e^{\gamma|x|^p}f\right\|_{L^2},
\end{split}
\end{equation*}
where $C_{m,n}=\int_{\mathbb{R}^n}e^{-\frac12|x|^p}dx$. Thus \eqref{eq2.3} holds for $\Theta_{A,\,B}(\gamma)=a$, $N_1=CC_{m,n}(\frac2c)^{(2m-1)n/2m}$ and $N_2=(\frac2c)^{2m-1}$.
\end{proof}

For $u$ in Proposition \ref{prop2.4} and any $\epsilon>0$, we define
\begin{equation}\label{eq2.7}
u_\epsilon(t)=e^{-\epsilon((-\Delta)^m+V)}u(t)=e^{-(\epsilon+it)((-\Delta)^m+V)}u(0),\quad t\in\mathbb{R}.
\end{equation}
By semigroup analyticity, we have $u_\epsilon\in C^\infty(\mathbb{R};H^{2m}(\mathbb{R}^n))$ and
\begin{equation*}\label{eq2.8}
i\partial_tu_\epsilon=((-\Delta_x)^m+V)u_\epsilon.
\end{equation*}
We still use $p=\frac{2m}{2m-1}$ in the next lemma.

\begin{lemma}\label{lm2.2}
There exists $\gamma_\epsilon>0$ such that
\begin{equation}\label{eq2.9}
\sum_{|\nu|\leq 2m-1}\sup_{0\leq t\leq1}\left\|e^{\gamma_\epsilon|x|^p}\partial_x^\nu u_\epsilon(t,x)\right\|_{L_x^2}+\sup_{0\leq t\leq1}\left\|e^{\gamma_\epsilon|x|^p}(-\Delta_x)^mu_\epsilon(t,x)\right\|_{L_x^2}<\infty.
\end{equation}
\end{lemma}

\begin{proof}[\textbf{Proof}]
Applying Lemma \ref{lm2.1} with $A=\epsilon$ and $B=t$, we have, with
\begin{equation*}
\gamma^{(0)}=\gamma(1+N_2\epsilon(1+\epsilon^{-2})^m\gamma^{2m-1})^{-\frac{1}{2m-1}},
\end{equation*}
that
\begin{equation}\label{eq2.10}
\sup_{0\leq t\leq1}\left\|e^{\gamma^{(0)}|x|^p}u_\epsilon(t,x)\right\|_{L_x^2}<\infty.
\end{equation}
If $k\leq2m-2$ and
\begin{equation*}
\sum_{|\nu|=k}\sup_{0\leq t\leq1}\|e^{\gamma^{(k)}|x|^p}\partial_x^\nu u_\epsilon(t,x)\|_{L_x^2}<\infty,
\end{equation*}
then for any $j=1,\cdots,n$, by taking $\gamma^{(k+1)}<\frac{\gamma^{(k)}}{2}$ we have
\begin{equation}\label{eq2.11}
\begin{split}
&\left\|e^{\gamma^{(k+1)}|x|^p}\partial_{x_j}\partial_x^\nu u_\epsilon(t,x)\right\|^2_{L_x^2}\\
=&-\int_{\mathbb{R}^n} e^{2\gamma^{(k+1)}|x|^p}\partial_x^\nu u_\epsilon\overline{\partial_{x_j}^2\partial_x^\nu u_\epsilon}dx-2p\gamma^{(k+1)}\int_{\mathbb{R}^n}|x|^{p-2}x_je^{2\gamma^{(k+1)}|x|^p}\partial_x^\nu u_\epsilon\overline{\partial_{x_j}\partial_x^\nu u_\epsilon}dx\\
\leq&C\sup_{0\leq t\leq1}\left(\left\|u_\epsilon(t,x)\right\|_{H_x^{2m}}\cdot\left\|e^{\gamma^{(k)}|x|^p}\partial_x^\nu u_\epsilon(t,x)\right\|_{L_x^2}\right),
\end{split}
\end{equation}
which inductively implies
\begin{equation}\label{eq2.12}
\sum_{|\nu|\leq2m-1}\sup_{0\leq t\leq1}\left\|e^{\gamma^{(2m-1)}|x|^p}\partial_x^\nu u_\epsilon(t,x)\right\|_{L_x^2}<\infty.
\end{equation}
Finally we take $\gamma_\epsilon<\frac{\gamma^{(2m-1)}}{2}$, and by \eqref{eq2.10} we have
\begin{equation*}
\begin{split}
2\left\|e^{\gamma_\epsilon|x|^p}(-\Delta_x)^mu_\epsilon\right\|^2_{L_x^2}\leq&\left\|e^{\gamma_\epsilon|x|^p}((-\Delta_x)^m+V)u_\epsilon\right\|^2_{L_x^2}+C\\
=&-i\int_{\mathbb{R}^n} e^{2\gamma_\epsilon|x|^p}(-\Delta_x)^mu_\epsilon\overline{\partial_tu_\epsilon}dx-i\int_{\mathbb{R}^n} e^{2\gamma_\epsilon|x|^p}Vu_\epsilon\overline{\partial_tu_\epsilon}dx+C.
\end{split}
\end{equation*}
Similar to \eqref{eq2.11}, we integrate by parts, use \eqref{eq2.12} and the regularity of $((-\Delta_x)^m+V)u_\epsilon$ to obtain \eqref{eq2.9}.
\end{proof}

The following abstract lemma for logarithmic convexity was proved in \cite{EKPV-MRL08,EKPV-JEMS08}.

\begin{lemma}\label{lm2.3}
Suppose $\mathcal{S}$ and $\mathcal{A}$ are symmetric and anti-symmetric differential operators in $\mathbb{R}^n$ with time-independent smooth bounded coefficients. Then
\begin{equation*}
\partial_t\frac{(\mathcal{S}g,g)}{||g||^2}\geq\frac{(\mathcal{S}g,\mathcal{A}g)+(\mathcal{A}g,\mathcal{S}g)-\frac12||\partial_tg-\mathcal{A}g-\mathcal{S}g||^2}{||g||^2},\quad t\in(0,T),
\end{equation*}
holds for all $g\in C^\infty([0,T];\mathscr{S}(\mathbb{R}^n))$ with $g(t)\neq0$ for all $t\in[0,T]$, where $\|\cdot\|=\|\cdot\|_{L_x^2(\mathbb{R}^n)}$.
\end{lemma}

Now we are ready to prove Proposition \ref{prop2.4}.

\begin{proof}[Proof of Proposition \ref{prop2.4}]
Since \eqref{eq2.15} is scaling invariant, we may just assume $T=1$. We use $p=\frac{2m}{2m-1}$ in the proof.

First take a decreasing $\theta\in C^\infty(\mathbb{R};[0,1])$ such that $\theta(s)=1$ when $s\leq1$ and $\theta(s)=0$ when $s\geq2$. Define $\varphi_R(x_1)=\int_{0}^{x_1}\theta(s/R)ds$ for $R>1$, we have $\varphi_R(x_1)\uparrow x_1$ when $R\rightarrow+\infty$. Next, for any $\epsilon\in(0,1)$ and $\lambda\in\mathbb{R}$, we define
\begin{equation*}
f_{\epsilon,\,\lambda,\,R}(t,x)=e^{\lambda\varphi_R(x_1)}u_\epsilon(t,x),
\end{equation*}
where $u_\epsilon$ is defined by \eqref{eq2.7}. Since $\varphi_R$ and all its derivatives are bounded, we know that $f_{\epsilon,\,\lambda,\,R}\in C^\infty([0,1];H^{2m}(\mathbb{R}^n))$. If we employ the notations $f_R=f_{\epsilon,\,\lambda,\,R}$, $P(D_x)=(-\Delta_x)^m$ for convenience, where $D_x=i^{-1}\nabla_x$ and $P(\xi)=|\xi|^{2m}$, we have
\begin{equation}\label{eq2.17}
\partial_tf_R=-i(P(D_x+i\lambda\theta_R(x_1)e_1)+V)f_R=\mathcal{S}_Rf_R+\mathcal{A}_Rf_R-iVf_R,\quad(t,x)\in(0,1)\times\mathbb{R}^n,
\end{equation}
where $\theta_R(x_1)=\theta(x_1/R)$, $e_1$ is the first unit vector in $\mathbb{R}^n$, and
\begin{equation*}
\begin{cases}
\mathcal{S}_R=\frac{1}{2i}(P(D_x+i\lambda\theta_R(x_1)e_1)-P(D_x-i\lambda\theta_R(x_1)e_1)),\\
\mathcal{A}_R=\frac{1}{2i}(P(D_x+i\lambda\theta_R(x_1)e_1)+P(D_x-i\lambda\theta_R(x_1)e_1))=i^{-1}(-\Delta_x)^m+\mathrm{L.O.T.}.
\end{cases}
\end{equation*}
Here $\mathcal{S}_R$ is symmetric with order $2m-1$, $\mathcal{A}_R$ is anti-symmetric with order $2m$, and they both have time-independent smooth bounded coefficients. From now on we assume $f_R(t,\cdot)\neq0$ for $(t,R)\in[0,1]\times(1,+\infty)$, otherwise it is easy to see that $u\equiv0$. We may also assume that
\begin{equation}\label{eq219}
\inf_{(t,R)\in[0,1]\times(1,+\infty)}\|f_R(t,x)\|_{L_x^2}=M\geq1,
\end{equation}
otherwise consider $M^{-1}u_\epsilon$. Now apply Lemma \ref{lm2.3} with $\mathcal{S}=\mathcal{S}_R$, $\mathcal{A}=\mathcal{A}_R$, $g=g_j\in C^\infty([0,1];\mathscr{S}(\mathbb{R}^n))$ with $g_j\rightarrow f_R$ in $L^1([0,1];H^{2m}(\mathbb{R}^n))$, we have by \eqref{eq2.17} that for all $\phi\in C_c^\infty((0,1);\mathbb{R}_+)$,
\begin{equation}\label{eq2.19}
\begin{split}
-\int_{0}^{1}\frac{(\mathcal{S}_Rf_R,f_R)}{\|f_R\|^2_{L_x^2}}\partial_t\phi dt\geq&\int_{0}^{1}\frac{(\mathcal{S}_Rf_R,\mathcal{A}_Rf_R)+(\mathcal{A}_Rf_R,\mathcal{S}_Rf_R)-\frac12\|Vf_R\|^2_{L^2}}{\|f_R\|^2_{L_x^2}}\phi dt\\
\geq&\int_{0}^{1}\frac{(\mathcal{S}_Rf_R,\mathcal{A}_Rf_R)+(\mathcal{A}_Rf_R,\mathcal{S}_Rf_R)}{\|f_R\|^2_{L_x^2}}\phi dt-\frac{\|V\|^2_{L^\infty}}{2}\int_{0}^{1}\phi dt.
\end{split}
\end{equation}
Since $V$ is real valued, it follows from \eqref{eq2.17} that
\begin{equation*}
\partial_t\log\|f_R\|^2_{L_x^2}=\frac{2(\mathcal{S}_Rf_R,f_R)}{\|f_R\|^2_{L_x^2}},
\end{equation*}
therefore \eqref{eq2.19} reads
\begin{equation}\label{eq2.20}
\int_{0}^{1}\log\|f_R\|^2_{L_x^2}\partial_t^2\phi dt\geq\int_{0}^{1}\frac{(\mathcal{S}_Rf_R,\mathcal{A}_Rf_R)+(\mathcal{A}_Rf_R,\mathcal{S}_Rf_R)}{\|f_R\|^2_{L_x^2}}\phi dt-\|V\|^2_{L^\infty}\int_{0}^{1}\phi dt.
\end{equation}
In order to consider sending $R\rightarrow+\infty$ in \eqref{eq2.20}, we first note that
\begin{equation*}
\begin{cases}
P(D_x+i\lambda\theta_R(x_1)e_1)f_R=e^{\lambda\varphi_R(x_1)}P(D_x)u_\epsilon,\\
P(D_x-i\lambda\theta_R(x_1)e_1)f_R=e^{-\lambda\varphi_R(x_1)}P(D_x)e^{2\lambda\varphi_R(x_1)}u_\epsilon=e^{\lambda\varphi_R(x_1)}P(D_x-2i\lambda\theta_R(x_1)e_1)u_\epsilon.
\end{cases}
\end{equation*}
Since $\theta_R$ and all its derivatives are bounded uniformly in $R$, we have
\begin{equation}\label{eq2.21}
\begin{split}
|\mathcal{S}_Rf_R|\leq& C_\lambda\sum_{|\nu|\leq2m-1}e^{|\lambda x_1|}|\partial_x^\nu u_\epsilon|,\\
|\mathcal{A}_Rf_R|\leq&C_\lambda\left(e^{|\lambda x_1|}\left|(-\Delta)^mu_\epsilon\right|+\sum_{|\nu|\leq2m-1}e^{|\lambda x_1|}|\partial_x^\nu u_\epsilon|\right)\\
\leq&C_\lambda e^{|\lambda x_1|}\left|((-\Delta)^m+V)u_\epsilon\right|+C_{\lambda,\,V}\sum_{|\nu|\leq2m-1}e^{|\lambda x_1|}|\partial_x^\nu u_\epsilon|.
\end{split}
\end{equation}
On the other hand, we have the almost everywhere convergence as $R\rightarrow+\infty$:
\begin{equation}\label{eq2.22}
\begin{cases}
\mathcal{S}_Rf_R\rightarrow\frac{1}{2i}(P(D_x+i\lambda e_1)-P(D_x-i\lambda e_1))f_{\epsilon,\,\lambda},\\
\mathcal{A}_Rf_R\rightarrow\frac{1}{2i}(P(D_x+i\lambda e_1)+P(D_x-i\lambda e_1))f_{\epsilon,\,\lambda},
\end{cases}
\end{equation}
where $f_{\epsilon,\,\lambda}=e^{\lambda x_1}u_\epsilon$. By Lemma \ref{lm2.2}, \eqref{eq2.21}, \eqref{eq2.22} and dominated convergence, we let $R\rightarrow+\infty$ in \eqref{eq2.20}. Recall the assumption \eqref{eq219}, we have
\begin{equation*}
\int_{0}^{1}\log\|f_{\epsilon,\,\lambda}\|^2_{L_x^2}\partial_t^2\phi dt\geq-\|V\|^2_{L^\infty}\int_{0}^{1}\phi dt=\frac{\|V\|^2_{L^\infty}}{2}\int_{0}^{1}t(1-t)\partial_t^2\phi dt.
\end{equation*}
This just means that the distribution
\begin{equation*}
t\mapsto\log\|f_{\epsilon,\,\lambda}(t,x)\|^2_{L_x^2}-\|V\|^2_{L^\infty}\frac{t(1-t)}{2}
\end{equation*}
is convex in $[0,1]$. In the other words,
\begin{equation*}
\left\|e^{\lambda x_1}u_\epsilon(t,x)\right\|^2_{L_x^2}\leq e^{\frac{t(1-t)}{2}\|V\|^2_{L^\infty}}\left\|e^{\lambda x_1}u_\epsilon(0,x)\right\|^{2(1-t)}_{L_x^2}\left\|e^{\lambda x_1}u_\epsilon(1,x)\right\|^{2t}_{L_x^2},\quad t\in[0,1],~\lambda\in\mathbb{R}.
\end{equation*}
By the rotation symmetry of $(-\Delta_x)^m$, we also have
\begin{equation}\label{eq2.26}
\left\|e^{\lambda\cdot x}u_\epsilon(t,x)\right\|^2_{L_x^2}\leq e^{\frac{t(1-t)}{2}\|V\|^2_{L^\infty}}\left\|e^{\lambda\cdot x}u_\epsilon(0,x)\right\|^{2(1-t)}_{L_x^2}\left\|e^{\lambda\cdot x}u_\epsilon(1,x)\right\|^{2t}_{L_x^2},\quad t\in[0,1],~\lambda\in\mathbb{R}^n,
\end{equation}
where we have abused the dimensionality of $\lambda$.

Recall the subordination inequality \eqref{eq2.27}. If we replace $x$ by $(2p\Theta_{\epsilon,0}(\gamma))^\frac1px$ in \eqref{eq2.27}, replace $\lambda$ by $\frac12(2p\Theta_{\epsilon,0}(\gamma))^\frac1p\lambda$ in \eqref{eq2.26}, multiply both sides of \eqref{eq2.26} by $e^{-|\lambda|^q/q}|\lambda|^{n(q-2)/2}$ and integrate over $\mathbb{R}^n$ with respect to $\lambda$, we have by H\"{o}lder's inequality that
\begin{equation}\label{eq2.28}
\begin{split}
\left\|e^{\Theta_{\epsilon,\,0}|x|^p}u_\epsilon(t,x)\right\|^2_{L_x^2}\leq&e^{\frac{t(1-t)}{2}\|V\|^2_{L^\infty}}\left\|e^{\Theta_{\epsilon,\,0}(\gamma)|x|^p}u_\epsilon(0,x)\right\|^{2(1-t)}_{L_x^2}\left\|e^{\Theta_{\epsilon,\,0}(\gamma)|x|^p}u_\epsilon(1,x)\right\|^{2t}_{L_x^2}\\
\leq&Ce^{\frac{t(1-t)}{2}\|V\|^2_{L^\infty}}\left\|e^{\gamma|x|^p}u(0,x)\right\|^{2(1-t)}_{L_x^2}\left\|e^{\gamma|x|^p}u(1,x)\right\|^{2t}_{L_x^2},
\end{split}
\end{equation}
for $t\in[0,1]$, where the last line comes from $u_\epsilon(j)=e^{-\epsilon((-\Delta)^m+V)}u(j)$ for $j=0,1$, and Lemma \ref{lm2.1}. Finally, if we truncate $u_\epsilon(t,x)$ in the left hand side of \eqref{eq2.28} by $\chi_R(x)$ with $\mathrm{supp}\,\chi_R$ contained in a ball with radius $R$, by first letting $\epsilon\rightarrow0$ and then $R\rightarrow+\infty$, \eqref{eq2.15} is proved when $T=1$.
\end{proof}

\subsection{Carleman estimate in spatial dimension one}\label{sb2.2}\ 

We shall prove Proposition \ref{lm2.6} in this part. The main tool is the Tr\`{e}ves identity (see \cite[Lemma 17.2.2]{Hor3}), and the following is a special case in dimension one.

\begin{lemma}\label{lm2.5}
Let $Q(x)=ax+\frac b2x^2+c$ be a real quadratic function in $\mathbb{R}$, and $P$ be a polynomial in $\mathbb{R}$ with real constant coefficients. Then for all $u\in C_c^\infty(\mathbb{R})$, denoted by $v=e^{Q/2}u$, we have
\begin{equation*}\label{eq2.29}
\int_\mathbb{R} e^Q|P(D_x)u|^2dx=\int_\mathbb{R}|P(D_x-D_xQ/2)v|^2dx=\sum_{k\geq0}\frac{b^k}{k!}\int_\mathbb{R}\left|P^{(k)}(D_x+D_xQ/2)v\right|^2dx,
\end{equation*}
where $P^{(k)}$ is the $k$-th derivative of $P$ and the summation above is obviously finite.
\end{lemma}

\begin{proof}[Proof of Proposition \ref{lm2.6}]
Denoted by $p=\frac{2m}{2m-1}$, $\alpha=\gamma R^p$ and $v=e^{Q/2}u$, we first write
\begin{equation}\label{eq2.31}
\begin{split}
\iint e^Q\left|D_tu+D_x^{2m}u\right|^2dxdt=&\iint\left|(D_t-D_tQ/2)v+(D_x-D_xQ/2)^{2m}v\right|^2dxdt\\
=&I+II+III,
\end{split}
\end{equation}
where
\begin{equation*}
\begin{split}
&I=\iint\left|(D_x-D_xQ/2)^{2m}v\right|^2dxdt,\quad II=\iint\left|(D_t-D_tQ/2)v\right|^2dxdt,\\
&III=2\mathrm{Re}\iint(D_t-D_tQ/2)v\overline{(D_x-D_xQ/2)^{2m}v}dxdt.
\end{split}
\end{equation*}
To treat $I$, notice that $Q=\frac{4\alpha\phi(t)}{R}x+\frac{4\alpha/R^2}{2}x^2+2\alpha(\phi(t))^2$, we apply Lemma \ref{lm2.5} in the spatial integral to have
\begin{equation}\label{eq2.33}
\begin{split}
I=&\sum_{k=0}^{2m}\frac{(4\alpha/R^2)^k}{k!}\iint\left|\frac{(2m)!}{(2m-k)!}(D_x+D_xQ/2)^{2m-k}v\right|^2dxdt\\
\geq&\iint\left|(D_x+D_xQ/2)^{2m}v\right|^2dxdt+\frac{16m^2\alpha}{R^2}\iint\left|(D_x+D_xQ/2)^{2m-1}v\right|^2dxdt.
\end{split}
\end{equation}
To treat II, notice that the commutator
\begin{equation*}
\mbox{$[D_t+D_tQ/2,D_t-D_tQ/2]=\partial_t^2Q=4\alpha(\phi'(t))^2+4\alpha\phi''(t)\left(\frac xR+\phi(t)\right),$}
\end{equation*}
by using $|\frac xR+\phi(t)|\leq d_2$ in $\mathrm{supp}\,v$, we have from integration by parts that
\begin{equation}\label{eq2.35}
II\geq\iint\left|(D_t+D_tQ/2)v\right|^2dxdt-K\alpha\iint|v|^2dxdt,
\end{equation}
where
\begin{equation*}
K=4\sup_{t\in[0,1]}\left((\phi'(t))^2+d_2|\phi''(t)|\right).
\end{equation*}
To treat III, since $[D_x+D_xQ/2,D_t-D_tQ/2]=\partial_x\partial_tQ=\frac{4\alpha}{R}\phi'(t)$ and consequently
\begin{equation*}
\begin{split}
&[(D_x+D_xQ/2)^{2m},D_t-D_tQ/2]\\
=&\sum_{k=0}^{2m-1}(D_x+D_xQ/2)^{2m-1-k}[D_x+D_xQ/2,D_t-D_tQ/2](D_x+D_xQ/2)^k\\
=&\frac{8m\alpha\phi'(t)}{R}(D_x+D_xQ/2)^{2m-1},
\end{split}
\end{equation*}
integration by parts and Cauchy-Schwarz inequality then give
\begin{equation}\label{eq2.37}
\begin{split}
III=&2\mathrm{Re}\iint(D_x+D_xQ/2)^{2m}v\overline{(D_t+D_tQ/2)v}dxdt\\
&\quad\quad\quad+ \frac{16m\alpha}{R}\iint(D_x+D_xQ/2)^{2m-1}v\overline{\phi'(t)v}dxdt\\
\geq&2\mathrm{Re}\iint(D_x+D_xQ/2)^{2m}v\overline{(D_t+D_tQ/2)v}dxdt\\
&\quad-\frac{8m^2\alpha}{R^2}\iint\left|(D_x+D_xQ/2)^{2m-1}v\right|^2dxdt-8\alpha\iint|v|^2dxdt.
\end{split}
\end{equation}
Combining \eqref{eq2.31}, \eqref{eq2.33}, \eqref{eq2.35} and \eqref{eq2.37} gives
\begin{equation}\label{eq2.38}
\begin{split}
\iint e^Q\left|D_tu+D_x^{2m}u\right|^2dxdt\geq&\iint\left|(D_t+D_tQ/2)v+(D_x+D_xQ/2)^{2m}v\right|^2dxdt\\
&+\frac{8m^2\alpha}{R^2}\iint\left|(D_x+D_xQ/2)^{2m-1}v\right|^2dxdt\\
&-(K+8)\alpha\iint|v|^2dxdt.
\end{split}
\end{equation}

Next we study the lower bound of the second line on the right hand side of \eqref{eq2.38}. Notice that
\begin{equation*}
\mbox{$(D_x-D_xQ/2)(D_x+D_xQ/2)=D_x^2+\frac{4\alpha^2}{R^2}\left(\frac xR+\phi(t)\right)^2-\frac{2\alpha}{R^2},$}
\end{equation*}
if $\gamma_0$ is known and $R_0=(d_1^2\gamma_0)^{-\frac1p}$, since $|\frac xR+\phi(t)|\geq d_1$ in $\mathrm{supp}\,v$, we have $2d_1^2\alpha-1\geq d_1^2\alpha$ when $\gamma\geq\gamma_0$ and $R\geq R_0$, and thus
\begin{equation}\label{eq2.40}
\begin{split}
&\iint\left|(D_x+D_xQ/2)^{2m-1}v\right|^2dxdt\\
=&\iint\left|D_x(D_x+D_xQ/2)^{2m-2}v\right|^2dxdt+\frac{4\alpha^2}{R^2}\iint\left(\frac xR+\phi(t)\right)^2\left|(D_x+D_xQ/2)^{2m-2}v\right|^2dxdt\\
&\quad-\frac{2\alpha}{R^2}\iint\left|(D_x+D_xQ/2)^{2m-2}v\right|^2dxdt\\
\geq&\frac{2d_1^2\alpha^2}{R^2}\iint\left|(D_x+D_xQ/2)^{2m-2}v\right|^2dxdt\\
\geq&\left(\frac{2d_1^2\alpha^2}{R^2}\right)^{2m-1}\iint|v|^2dxdt.
\end{split}
\end{equation}
Combining \eqref{eq2.38}, \eqref{eq2.40} and the notation $\alpha=\gamma R^{2m/(2m-1)}$, we have
\begin{equation*}
\iint e^Q\left|D_tu+D_x^{2m}u\right|^2dxdt\geq\left(8m^2(2d_1^2)^{2m-1}\gamma^{4m-2}-K-8\right)\gamma R^\frac{2m}{2m-1}\iint e^Q|u|^2dxdt.
\end{equation*}
Now we can choose $\gamma_0$ large enough to complete the proof of \eqref{eq2.30}.
\end{proof}

\subsection{Proof of uniqueness}\label{sb2.3}\ 

\begin{proof}[Proof of Theorem \ref{thm1.1}]
By translation and scaling, we may assume \eqref{eq1.3} with $T_1=0$, $T_2=1$, and $\|V\|_{L^\infty}=O(T)$. We shall see in the proof that the $\tilde{\gamma}$ will be found independent of $T>0$. We still denote by $p=\frac{2m}{2m-1}$ in the proof. 

First consider $u_{\epsilon,\,\gamma}(t)=e^{-\epsilon\gamma^{-(2m-1)}(D_x^{2m}+V)}u(t)$ for $\epsilon\in(0,1)$, and we apply Lemma \ref{lm2.1} with $A=\epsilon\gamma^{-(2m-1)}$, $B=0$ and $f=u(t,\cdot)$. Then
\begin{equation}\label{eq2.42}
\begin{cases}
u_{\epsilon,\,\gamma}\in C^\infty(\mathbb{R};H^{2m}(\mathbb{R})),\\
i\partial_tu_{\epsilon,\,\gamma}=(D_x^{2m}+V)u_{\epsilon,\,\gamma},\quad(t,x)\in\mathbb{R}^2,\\
\sup\limits_{0\leq t\leq1}\sum\limits_{k=0}\limits^{2m-1}\left\|e^{\frac{\gamma}{8^m(1+N_2)}|x|^p}\partial_x^ku_{\epsilon,\,\gamma}(t,x)\right\|_{L_x^2(\mathbb{R})}<C_{\epsilon,V}<\infty.
\end{cases}
\end{equation}
The conclusion in the last line of \eqref{eq2.42} for $k=0$ is based on Proposition \ref{prop2.4} and the fact (see \eqref{eq2.4}) that
\begin{equation*}
\Theta_{\epsilon\gamma^{-(2m-1)},\,0}(\gamma)=\frac{\gamma}{(1+N_2\epsilon)^\frac{1}{2m-1}}\geq\frac{\gamma}{1+N_2}.
\end{equation*}
The other cases are shown following the proof of Lemma \ref{lm2.2}. In the sequel, we let
\begin{equation}\label{2.24}
\gamma'=\frac{\gamma}{8^m(1+N_2)}.
\end{equation}

Next, we take $\eta\in C_c^\infty((0,1);[0,1])$ such that $\eta\equiv1$ on $[\frac14,\frac34]$; and take $\theta\in C_c^\infty(\mathbb{R};[0,1])$ such that $\theta(x)=1$ when $|x|<\frac12$, $\theta(x)=0$ when $|x|>1$. Denoted by $\theta_R(\cdot)=\theta_R(\frac\cdot R)$, we define $U_{R,\,\epsilon,\,\gamma}(t,x)=\eta(t)\theta_R(x)u_{\epsilon,\,\gamma}(t,x)$. Then $U_{R,\,\epsilon,\,\gamma}\in C^\infty([0,1];H^{2m}(\mathbb{R}))$ satisfies
\begin{equation*}
D_tU_{R,\,\epsilon,\,\gamma}+D_x^{2m}U_{R,\,\epsilon,\,\gamma}=-VU_{R,\,\epsilon,\,\gamma}-i\partial_t\eta\theta_Ru_{\epsilon,\,\gamma}+\eta\sum_{k=0}^{2m-1}C_k\partial_x^{2m-k}\theta_R\partial_x^ku_{\epsilon,\,\gamma}.
\end{equation*}
Let $\phi(t)=-4(t-\frac12)^2+\frac94$. Since $\frac14\leq|\frac xR+\phi(t)|\leq\frac{13}{4}$ when $(t,x)\in\mathrm{supp}\,U_{R,\,\epsilon,\,\gamma}$, we can now apply Proposition \ref{lm2.6} to $U_{R,\,\epsilon,\,\gamma}$ with such $\phi(t)$.
Denoted by $Q=2\gamma_0R^p(\frac xR+\phi(t))^2$ for the $\gamma_0$ found in Proposition \ref{lm2.6}, we have when $R\geq R_0$ for some $R_0>0$ that
\begin{equation}\label{eq2.45}
\begin{split}
&C\gamma_0^{4m-1}R^p\iint e^Q|U_{R,\,\epsilon,\,\gamma}|^2dxdt\\
\leq&\iint e^Q\left|D_tU_{R,\,\epsilon,\,\gamma}+D_x^{2m}U_{R,\,\epsilon,\,\gamma}\right|^2dxdt\\
\leq&C'\|V\|^2_{L^\infty(\mathbb{R})}\iint e^Q|U_{R,\,\epsilon,\,\gamma}|^2dxdt+C'\iint_{\mathrm{supp}\,\partial_t\eta\times\mathrm{supp}\,\theta_R}e^Q|u_{\epsilon,\,\gamma}|^2dxdt\\
&+C'R^{-2}\sum_{k=0}^{2m-1}\iint_{\mathrm{supp}\,\eta\times\mathrm{supp}\,\partial_x^{2m-k}\theta_R}e^Q|\partial_x^ku_{\epsilon,\,\gamma}|^2dxdt.
\end{split}
\end{equation}
When $R^p\gg\|V\|_{L^\infty}^2=O(T^2)$, the potential term on the right hand side of \eqref{eq2.45} is absorbed into the left hand side. Then we can restrict the domain of integration on the left hand side to $(t,x)\in[\frac{7}{16},\frac{9}{16}]\times B(\frac{R}{64})$, where $U_{R,\,\epsilon,\,\gamma}\equiv u_{\epsilon,\,\gamma}$ and
\begin{equation*}
\mbox{$Q\geq2\gamma_0 R^p\cdot\left(-\frac{1}{64}+\frac94-\frac{1}{64}\right)^2=2\gamma_0R^p\cdot\left(\frac94-\frac{1}{32}\right)^2.$}
\end{equation*}
Therefore
\begin{equation}\label{eq2.47}
\begin{split}
&CR^pe^{2\gamma_0R^p\cdot\left(\frac94-\frac{1}{32}\right)^2}\iint_{[\frac{7}{16},\frac{9}{16}]\times B(\frac{R}{64})}|u_{\epsilon,\,\gamma}|^2dxdt\\
\leq&\iint_{\mathrm{supp}\,\partial_t\eta\times\mathrm{supp}\,\theta_R}e^Q|u_{\epsilon,\,\gamma}|^2dxdt+R^{-2}\sum_{k=0}^{2m-1}\iint_{\mathrm{supp}\,\eta\times\mathrm{supp}\,\partial_x^{2m-k}\theta_R}e^Q|\partial_x^ku_{\epsilon,\,\gamma}|^2dxdt.
\end{split}
\end{equation}

In $\mathrm{supp}\,\partial_t\eta\times\mathrm{supp}\,\theta_R$, it follows that $(t,x)\in[0,\frac14]\cup[\frac34,1]\times B(R)$, so if we take $\sigma$ with $(\frac94-\frac14)^2(1+\sigma)=(\frac94-\frac{1}{32})^2$, one has
\begin{equation*}
\begin{split}
Q\leq&\mbox{$2\gamma_0R^p\cdot\left(\left|\frac xR\right|+\frac94-\frac14\right)^2$}\\
\leq&\mbox{$2\gamma_0R^p\left(\frac xR\right)^2\left(1+\frac1\sigma\right)+2\gamma_0R^p\cdot\left(\frac94-\frac14\right)^2(1+\sigma)$}\\
\leq&\mbox{$2\gamma_0\left(1+\frac1\sigma\right)|x|^p+2\gamma_0R^p\cdot\left(\frac94-\frac{1}{32}\right)^2.$}
\end{split}
\end{equation*}
Therefor, if
\begin{equation}\label{eq2.49}
\mbox{$\gamma'\geq\left(1+\frac1\sigma\right)\gamma_0,$}
\end{equation}
we have
\begin{equation}\label{eq2.50}
\iint_{\mathrm{supp}\,\partial_t\eta\times\mathrm{supp}\,\theta_R}e^Q|u_{\epsilon,\,\gamma}|^2dxdt\leq e^{2\gamma_0R^p\cdot\left(\frac94-\frac{1}{32}\right)^2}\sup_{0\leq t\leq1}\left\|e^{\gamma'|x|^p}u_{\epsilon,\,\gamma}(t,x)\right\|^2_{L_x^2(\mathbb{R})}.
\end{equation}

In $\mathrm{supp}\,\eta\times\mathrm{supp}\,\partial_x^{2m-k}\theta_R$ it follows that $(t,x)\in[0,1]\times(B(R)\setminus B(R/2))$, then
\begin{equation*}
\begin{split}
Q\leq&\mbox{$2\gamma_0R^p\cdot\left(1+\frac94\right)^2$}\\
=&\mbox{$2\gamma'|x|^p-\left(2\gamma'|x|^p-2\gamma_0R^p\cdot\left(\frac{13}{4}\right)^2\right)$}\\
\leq&\mbox{$2\gamma'|x|^p-2\left(\frac{\gamma'}{2^p}-\left(\frac{13}{4}\right)^2\gamma_0\right)R^p.$}
\end{split}
\end{equation*}
Thus when
\begin{equation}\label{eq2.52}
\mbox{$\gamma'\geq2^p\cdot\left(\frac{13}{4}\right)^2\gamma_0,$}
\end{equation}
we have
\begin{equation}\label{eq2.53}
\sum_{k=0}^{2m-1}\iint_{\mathrm{supp}\,\eta\times\mathrm{supp}\,\partial_x^{2m-k}\theta_R}e^Q|\partial_x^ku_{\epsilon,\,\gamma}|^2dxdt\leq C\sup_{0\leq t\leq1}\sum_{k=0}^{2m-1}\left\|e^{\gamma'|x|^p}\partial_x^ku_{\epsilon,\,\gamma}(t,x)\right\|^2_{L_x^2(\mathbb{R})}.
\end{equation}

Combining \eqref{eq2.47}, \eqref{eq2.50}, \eqref{eq2.53} and \eqref{eq2.42}, we get
\begin{equation}\label{2.31}
\begin{split}
CR^pe^{2\gamma_0R^p\cdot\left(\frac94-\frac{1}{32}\right)^2}\iint_{[\frac{7}{16},\frac{9}{16}]\times B(\frac{R}{64})}|u_{\epsilon,\,\gamma}|^2dxdt\leq C_{\epsilon,V}\left(e^{2\gamma_0R^p\cdot\left(\frac94-\frac{1}{32}\right)^2}+R^{-2}\right),
\end{split}
\end{equation}
for large $R$. Eliminate the exponential in\eqref{2.31} and let $R\rightarrow+\infty$, $u_{\epsilon,\gamma}\equiv0$ follows on $[\frac{7}{16},\frac{9}{16}]\times\mathbb{R}$ and thus on $[0,1]\times\mathbb{R}$, because
\begin{equation*}
u_{\epsilon,\gamma}(t)=e^{-\epsilon\gamma^{-(2m-1)}(D_x^{2m}+V)}u(t)=e^{-it(D_x^{2m}+V)}\left(e^{-\epsilon\gamma^{-(2m-1)}(D_x^{2m}+V)}u(0)\right).
\end{equation*}
The choices \eqref{eq2.49} and \eqref{eq2.52} for the largeness of $\gamma'$, together with \eqref{2.24} suggest that $\tilde{\gamma}$ is found independent of $\epsilon$ and $V$, so we may let $\epsilon\rightarrow0$ to complete the proof.
\end{proof}

\subsection{Constructions for Examples \ref{ex12}}\label{sb2.4}\ 

For \eqref{1.5}, we first take $f(x)=e^{-2|x|^{2m/(2m-1)}}$, $x\in\mathbb{R}^n$. By the strong continuity of $\{e^{-t((-\Delta_x)^m+V)}\}_{t\geq0}$ on $L^2(\mathbb{R}^n)$, there exists $\epsilon>0$ such that $e^{-\epsilon((-\Delta_x)^m+V)}f\neq0$. Let
\begin{equation*}
u=e^{-(\epsilon+it)((-\Delta_x)^m+V)}f,
\end{equation*}
then $u\in C^\infty(\mathbb{R};H^{2m}(\mathbb{R}^n))$ by semigroup analyticity, $\|u(t,x)\|_{L_x^2}\equiv\|e^{-\epsilon((-\Delta_x)^m+V)}f\|_{L^2}\neq0$, and $u$ solves \eqref{eq14}. We can apply Lemma \ref{lm2.1} with $A=\epsilon$, $B=t$ and $\gamma=1$ to have
\begin{equation*}
\left\|e^{\Theta{\epsilon,\,t}(1)|x|^\frac{2m}{2m-1}}u(t,x)\right\|_{L_x^2}\leq N_1e^{\omega_0\epsilon\|V\|_{\infty}}\left(1+\epsilon^{-2}t^2\right)^\frac n2\left\|e^{|x|^\frac{2m}{2m-1}}f\right\|_{L^2}\leq C(1+|t|)^n.
\end{equation*}
Now $h(t)=\Theta{\epsilon,t}(1)$ is found (see \eqref{eq2.4}) decreasing in $|t|$, satisfying \eqref{e16}. A similar construction earlier found when $m=1$ was in \cite[Remark 1]{EKPV-JEMS08}, where the $(1+|t|)^n$-loss can be avoided by a more straightforward energy method.

For \eqref{1.6}, let $K(z,x)=\mathscr{F}^{-1}(e^{-z|\cdot|^{2m}})(x)$ be the convolution kernel of the analytic semigroup $\{e^{-z(-\Delta_x)^m}\}_{\mathrm{Re}\,z>0}$. If
\begin{equation*}
u(t,x)=K(1+it,x)=\left(e^{-it(-\Delta_x)^m}K(1,\cdot)\right)(x),\quad t\in\mathbb{R},~x\in\mathbb{R}^n,
\end{equation*}
then $u$ is obviously a non-trivial analytic solution to the free equation, and \eqref{1.6} follows by \eqref{eq2.2}.


\section{$L^p$ global unique continuation}\label{sec3}

Till the end, we denote by $p=\frac{4m+2}{4m+1}$ and $p'=4m+2$ for convenience.


\subsection{Proof of uniqueness}\ 

For the sake of completeness, we first show how Lemma \ref{lem3.1} implies Theorem \ref{thm1.2} in a standard way.

\begin{proof}[Proof of Theorem \ref{thm1.2}]
By translation, we may assume $D=\{(t,x)\in \mathbb{R}^{2};\,\langle(t,x),v\rangle>0\}$ for some $v\in \mathbb{S}^1$. We are left to show that $u\equiv0$ in the strip
$$
S_{v,\,\rho}\triangleq\{(t,x)\in\mathbb{R}^{2};\,-\rho\leq \langle(t,x),v\rangle\leq0\}
$$
for some $\rho>0$. 

Denoted by $B=\{(t,x)\in\mathbb{R}^2;\,|(t,x)|<1\}$, we first take $\psi\in C_c^{\infty}(B)$ with $\int_B\psi dxdt=1$. We also take $\phi\in C_c^\infty(B)$ with $\phi(t,x)=1$ if $|(t,x)|<\frac12$. For $0<\epsilon<1$ and $R>1$, denoted by $\psi_\epsilon(t,x)=\epsilon^{-2}\psi(t/\epsilon,x/\epsilon)$ and $\phi_R(t,x)=\phi(t/R,x/R)$, we set $u_\epsilon=u\ast\psi_\epsilon$ and $u_{\epsilon,\,R}=\phi_R u_\epsilon$. Then $u_{\epsilon,\,R}\in C_c^{\infty}(\mathbb{R}^{2})$, and it follows by the vanishing of $u$ that
\begin{align*}
\mathrm{supp}\,u_{\epsilon}\subset \Gamma_{v,\,\epsilon}=\{(t,x)\in \mathbb{R}^{2};\,\langle(t,x),v\rangle\leq\epsilon\}.
\end{align*}
So we can apply Carleman estimate \eqref{equ3.1} to $u_{\epsilon,\,R}$ and get
\begin{align}\label{equ3.4}
\left\|e^{\lambda\langle(t,x),v\rangle}u_{\epsilon,\,R}\right\|_{L^{p'}(\Gamma_{v,\,\epsilon})}\leq C\left\|e^{\lambda\langle(t,x),v\rangle}(i\partial_t-D_x^{2m})u_{\epsilon,\,R}\right\|_{L^p(\Gamma_{v,\,\epsilon})},\quad\lambda\in\mathbb{R}.
\end{align}
Recall $u\in W^p$, then the Mihlin multiplier theorem (e.g. \cite{gra}) implies $D_x^ju\in L^p(\mathbb{R}^2)$ for $1<j<2m$. If we consider $\lambda>0$, since
$$
e^{\lambda\langle(t,x),v\rangle}\left|(i\partial_t-D_x^{2m})u_{\epsilon,\,R}-\phi_R(i\partial_t-D_x^{2m})u_{\epsilon}\right|\leq \frac{Ce^{\lambda\epsilon}}{R}\sum_{j=0}^{2m-1}|D_x^ju_\epsilon|,\quad(t,x)\in\Gamma_{v,\,\epsilon},
$$
we may let $R\rightarrow \infty$ in \eqref{equ3.4} to have
\begin{align}\label{equ3.5}
\left\|e^{\lambda\langle(t,x),v\rangle}u_{\epsilon}\right\|_{L^{p'}(\Gamma_{v,\,\epsilon})}\leq C\left\|e^{\lambda\langle(t,x),v\rangle}(i\partial_t-D_x^{2m})u_{\epsilon}\right\|_{L^p(\Gamma_{v,\,\epsilon})},\quad\lambda>0.
\end{align}
The boundedness of $e^{\lambda\langle(t,x),v\rangle}$ in $\Gamma_{v,\,\epsilon}$ and the regularity of $u$ also guarantee the convergence of \eqref{equ3.5} when $\epsilon\rightarrow0$. We then use equation \eqref{eq1.1} and the vanishing of $u$ to have
\begin{equation*}
\begin{split}
&\left\|e^{\lambda\langle(t,x),v\rangle}u\right\|_{L^{p'}(S_{v,\,\rho})}\\
\leq&C\left\|e^{\lambda\langle(t,x),v\rangle}(i\partial_t-D_x^{2m})u\right\|_{L^p(S_{v,\,\rho})}+C\left\|e^{\lambda\langle(t,x),v\rangle}(i\partial_t-D_x^{2m})u\right\|_{L^p(\langle(t,x),v\rangle\leq-\rho)}\\
\leq&C\|V\|_{L^\frac{2m+1}{2m}(S_{v,\,\rho})}\left\|e^{\lambda\langle(t,x),v\rangle}u\right\|_{L^{p'}(S_{v,\,\rho})}+Ce^{-\lambda\rho}\left\|(i\partial_t-D_x^{2m})u\right\|_{L^p(\mathbb{R}^{2})},
\end{split}
\end{equation*}
where the last line comes from the fact that $\frac 1p=\frac{2m}{2m+1}+\frac{1}{p'}$. If $\rho>0$ is so small that the first term on the last line above is absorbed into the left hand side, we have
\begin{align}\label{equ3.7}
\left\|e^{\lambda(\langle(t,x),v\rangle+\rho)}u\right\|_{L^{p'}(S_{v,\,\rho})}\leq C\left\|(i\partial_t-D_x^{2m})u\right\|_{L^p(\mathbb{R}^{2})}.
\end{align}
By shrinking the left hand side of \eqref{equ3.7} to integration on $S_{v,\,\rho-\epsilon}$ for any small $\epsilon>0$ where $e^{\lambda(\langle(t,x),v\rangle+\rho)}$ is lower bounded by $e^{\lambda\epsilon}$, we can send $\lambda\rightarrow+\infty$ to obtain $u=0$ in the strip $S_{v,\,\rho}$, which completes the proof.
\end{proof}


\subsection{"Uniform resolvent estimate"}\ 

\begin{proof}[Proof of Claim \ref{cl1.7}]
Replace $f$ by $e^{-it\mathrm{Re}\,z}f(t,x)$, the proof of \eqref{eq310} is reduced to the special case $z=i\beta$ for $\beta\in\mathbb{R}\setminus\{0\}$. Let
\begin{equation*}
(Tf)(t,x)=\int_{\mathbb{R}^{2}}\frac{e^{i\langle(t,x),(\tau,\xi)\rangle}\hat{f}(\tau,\xi)}{\tau+P(\xi)+i\beta}d\tau d\xi.
\end{equation*}
If $\tilde{f}(t,\cdot)$ is the Fourier transform of $f$ in the spatial variable, then
\begin{equation*}
(Tf)(t,x)=\int_{\mathbb{R}}e^{ix\xi}\int_\mathbb{R}e^{-isP(\xi)}a(s)\tilde{f}(t-s,\xi)dsd\xi,
\end{equation*}
where $a(s)=\int_\mathbb{R}\frac{e^{i\tau s}}{\tau+i\beta}d\tau=-2\pi iH(-\beta s)e^{\beta s}$ and $H$ is the Heaviside function. Clearly $\|a\|_{L^\infty}\leq 2\pi$. Now since $P(\xi)$ is real and has degree $2m$, we can apply the van der Corput lemma (see e.g. \cite[p. 332]{St}) to have the following estimate in the sense of oscillatory integral
\begin{equation}\label{eq313.1}
\left|\int_{\mathbb{R}}e^{-isP(\xi)+ix\xi}d\xi\right|\leq C|s|^{-\frac{1}{2m}},\quad s\in \mathbb{R}\setminus \{0\},~x\in\mathbb{R}.
\end{equation}
This and the fact that $\|e^{-isP(D_x)}\|_{L^2(\mathbb{R})-L^2(\mathbb{R})}=1$ imply the interpolation
\begin{equation}\label{eq314}
\left\|e^{-isP(D_x)}\right\|_{L^p(\mathbb{R})-L^{p'}(\mathbb{R})}\leq C |s|^{-\frac{1}{2m}(\frac 1p-\frac{1}{p'})},\quad s\in\mathbb{R}\setminus\{0\}.
\end{equation}
Thus we use \eqref{eq314} and  Minkowski's inequality to obtain
\begin{equation}\label{eq313}
\begin{split}
\|(Tf)(t,x)\|_{L_x^{p'}(\mathbb{R})}\leq&\int_\mathbb{R}\left\|\int_{\mathbb{R}}e^{ix\xi-isP(\xi)}\tilde{f}(t-s,\xi)d\xi\right\|_{L_x^{p'}(\mathbb{R})}ds\\
\leq&C\int_\mathbb{R}|s|^{-\frac{1}{2m}(\frac1p-\frac{1}{p'})}\|f(t-s,x)\|_{L_x^p(\mathbb{R})}ds,
\end{split}
\end{equation}
and \eqref{eq310} follows by applying the Hardy-Littlewood-Sobolev inequality to \eqref{eq313} using $1+\frac{1}{p'}=\frac{1}{p}+\frac{1}{2m}(\frac{1}{p}-\frac{1}{p'})$.
\end{proof}

\begin{remark}\label{rk3.1}
We have dropped a hint in the Introduction that a possible higher dimensional Carleman estimate in our context should be some generalization of the restriction estimate \eqref{e127}, and a natural form is
\begin{equation*}
\left\|e^{\lambda\langle(t,x),v\rangle}u\right\|_{L^\frac{4m+2n}{n}(\mathbb{R}^{1+n})}\leq C\left\|e^{\lambda\langle(t,x), v\rangle}(i\partial_t-(-\Delta_x)^m)u\right\|_{L^\frac{4m+2n}{4m+n}(\mathbb{R}^{1+n})},\quad u\in C_c^\infty(\mathbb{R}^{1+n}).
\end{equation*}
Following the same sketch, and using the rotation symmetry of $\Delta_x$, one might first be encountering the following uniform resolvent estimate
\begin{equation}\label{e37}
\left\|\mathscr{F}^{-1}\left(\frac{\hat{f}(\tau,\xi)}{\tau+P(\xi)+z}\right)\right\|_{L^\frac{4m+2n}{n}(\mathbb{R}^{1+n})}\leq C\|f\|_{L^\frac{4m+2n}{4m+n}(\mathbb{R}^{1+n})},\quad f\in C_c^\infty(\mathbb{R}^{1+n}),\,z\in\mathbb{C}\setminus\mathbb{R},
\end{equation}
and here $P(\xi)=\mathrm{Re}\,|\xi+ie_1|^{2m}=\mathrm{Re}\,(|\xi|^2-1+2i\xi_1)^m$ for example. If we roll the proof of Claim \ref{cl1.7} in this case, to coherently close the proof of \eqref{e37}, the analogue for the first estimate \eqref{eq313.1} should be
\begin{equation}\label{e38}
\left|\int_{\mathbb{R}^n}e^{-isP(\xi)+ix\cdot\xi}d\xi\right|\leq C|s|^{-\frac{n}{2m}},\quad s\in \mathbb{R}\setminus \{0\},~x\in\mathbb{R}^n.
\end{equation}

However, since \eqref{e38} is in the form of 'sharp dispersive estimate' typically for dispersive equations with non-degenerate Hamiltonian of order $2m$, \eqref{e38} is then rarely true. This is because the phase $P(\xi)$, whose leading term is $|\xi|^{2m}$, has many lower degree terms that are varying signs, which generically implies a degenerate situation. As already seen, such degeneracy is not a problem when $n=1$ by the van der Corput lemma. We also mention that for $P(\xi)$ with positive lower order terms, \eqref{e38} is true in many cases, see e.g. Huang et al. \cite{HHZ}.
\end{remark}


\subsection{$L^p$ Carleman estimate}\ 

As explained in the Introduction, to prove Lemma \ref{lem3.1}, we are left to show \eqref{eq39} by frequency localization using Claim \ref{cl1.7}. First recall \eqref{e130}, i.e. $M_b(\tau,\xi)=(\tau+P(\xi)+iQ_b(\xi))^{-1}$ where $P(\xi)=\mathrm{Re}\,(\xi+i)^{2m}$ and $Q_b(\xi)=\mathrm{Im}\,(\xi+i)^{2m}+b$. The real polynomials $P(\xi)$ and $Q_b(\xi)$ are of degrees $2m$ and $2m-1$ respectively, and we write
\begin{equation*}\label{eq319}
Q_b(\xi)=2m\prod_{j=1}^{2m-1}(\xi-\xi_{b,j}),\quad\xi\in\mathbb{R},~b\in\mathbb{R},
\end{equation*}
for some $\xi_{b,\,j}\in\mathbb{C}$, $j=1,\dots,2m-1$. Denoted by $a_j=\mathrm{Re}\,\xi_{b,j}$, we assume without loss of generality that
\begin{equation}\label{eq3.13}
a_1\leq\cdots\leq a_{2m-1}.
\end{equation}

\begin{proof}[Proof of Lemma \ref{lem3.1}]
For convenience, we divide the argument into 3 steps.

\noindent\textbf{Step 1. Frequency localization.}

Let $\chi_0$ be the characteristic function of interval $(a_1-\frac{|a_1|}{2},a_{2m-1}+\frac{|a_{2m-1}|}{2}]$, $\chi^+$ be the characteristic function of $(1,2]$, and $\chi^-$ be the characteristic function of $(-2,-1]$. Define
\begin{equation*}
\chi_k^+(\xi)=\chi^+\left(\frac{\xi-a_{2m-1}}{2^{k-2}|a_{2m-1}|}\right),\quad\chi_k^-(\xi)=\chi^-\left(\frac{\xi-a_1}{2^{k-2}|a_1|}\right),\quad k\geq1,
\end{equation*}
we have $\chi_0+\Sigma_{k\geq1}\chi_k^++\Sigma_{k\geq1}\chi_k^-\equiv1$. As earlier observed in \cite[p. 336]{KRS} for the case of second order hyperbolic operators, we conclude that \eqref{eq39} is a consequence of the following localized estimates: there is some constant $C>0$ independent of $k$ and $b$ such that
\begin{align}\label{eq315}
\left\|\mathscr{F}^{-1}(\chi_0(\xi)M_b(\tau, \xi)\hat{f}(\tau,\xi))\right\|_{L^{p'}(\mathbb{R}^{2})}\leq C\|\chi_0(D_x)f\|_{L^p(\mathbb{R}^{2})},
\end{align}
and
\begin{align}\label{eq315.b}
\left\|\mathscr{F}^{-1}(\chi_k^\pm(\xi)M_b(\tau, \xi)\hat{f}(\tau,\xi))\right\|_{L^{p'}(\mathbb{R}^{2})}\leq C\|\chi_k^\pm(D_x)f\|_{L^p(\mathbb{R}^{2})},\quad k\geq1.
\end{align}

By the one dimensional Littlewood-Paley theorem associated with non-smooth dyadic sums (see \cite[p. 349]{gra}), Minkowski's inequality with the fact that $p<2<p'$, \eqref{eq315} and \eqref{eq315.b}, we obtain
\begin{equation}\label{eq315.c}
\begin{split}
&\left\|\mathscr{F}^{-1}(M_b(\tau, \xi)\hat{f}(\tau,\xi))\right\|_{L^{p'}(\mathbb{R}^{2})}\\
\leq&\left\|\mathscr{F}^{-1}(\chi_0(\xi)M_b(\tau, \xi)\hat{f}(\tau,\xi))\right\|_{L^{p'}(\mathbb{R}^{2})}+C\left\|\left(\sum_{k=1}^{\infty}\left|\mathscr{F}^{-1}(\chi_k^+(\xi)M_b(\tau, \xi)\hat{f}(\tau,\xi))\right|^2\right)^{\frac12}\right\|_{L^{p'}(\mathbb{R}^{2})}\\
&+C\left\|\left(\sum_{k=1}^{\infty}\left|\mathscr{F}^{-1}(\chi_k^-(\xi)M_b(\tau, \xi)\hat{f}(\tau,\xi))\right|^2\right)^{\frac12}\right\|_{L^{p'}(\mathbb{R}^{2})}\\
\leq& C\left(\|\chi_0(D_x)f\|_{L^p(\mathbb{R}^{2})}+\left(\sum_{k=1}^{\infty}\|\chi_k^+(D_x)f\|^2_{L^p(\mathbb{R}^2)}\right)^{\frac12}+\left(\sum_{k=1}^{\infty}\|\chi_k^-(D_x)f\|^2_{L^p(\mathbb{R}^2)}\right)^{\frac12}\right)\\
\leq& C\left(\|\chi_0(D_x)f\|_{L^p(\mathbb{R}^{2})}+\left\|\left(\sum_{k=1}^{\infty}|\chi_k^+(D_x)f|^{2}\right)^{\frac12}\right\|_{L^p(\mathbb{R}^2)}+\left\|\left(\sum_{k=1}^{\infty}|\chi_k^-(D_x)f|^{2}\right)^{\frac12}\right\|_{L^p(\mathbb{R}^2)}\right)\\
\leq& C\|f\|_{L^p(\mathbb{R}^{2})}.\\
\end{split}
\end{equation}

We remark that $\chi_k^\pm$ has translation and scaling factors depending on $b$, but such operations for frequency cut-offs in the Littlewood-Paley theorem leave the same equivalence constants, and thus the constant $C$ in \eqref{eq315.c} is universal.

Now we are left to prove \eqref{eq315} and \eqref{eq315.b}.

\noindent\textbf{Step 2. Proof of \eqref{eq315.b}.} 

We only prove in the "+" case. By Claim \ref{cl1.7}, there is a constant $C>0$ independent of $k$ and $b$ such that
\begin{equation*}\label{eq316}
\left\|\mathscr{F}^{-1}\left(\frac{\chi_k^+(\xi)\hat{f}(\tau,\xi)}{\tau+P(\xi)+iQ_b(a_{2m-1,\,k})}\right)\right\|_{L^{p'}(\mathbb{R}^2)}\leq C\|\chi_k^+(D_x)f\|_{L^p(\mathbb{R}^2)},\quad k\geq1,
\end{equation*}
where $a_{2m-1,\,k}=a_{2m-1}+2^{k-2}|a_{2m-1}|$. This and the difference
\begin{equation*}\label{eq317}
\begin{split}
&\frac{1}{\tau+(\xi+i)^{2m}+ib}-\frac{1}{\tau+P(\xi)+iQ_b(a_{2m-1,\,k})}\\
=&\frac{i(Q_b(a_{2m-1,\,k})-Q_b(\xi))}{(\tau+P(\xi)+iQ_b(\xi))(\tau+P(\xi)+iQ_b(a_{2m-1,\,k}))},
\end{split}
\end{equation*}
indicate that \eqref{eq315.b} follows if one can prove the uniform estimates
\begin{equation*}\label{eq318}
\left\|\mathscr{F}^{-1}\left(\frac{(Q_b(a_{2m-1,\,k})-Q_b(\xi))\chi_k^+(\xi)\hat{f}(\tau,\xi)}{(\tau+P(\xi)+Q_b(\xi))(\tau+P(\xi)+iQ_b(a_{2m-1,\,k}))}\right)\right\|_{L^{p'}(\mathbb{R}^2)}\leq C\|\chi_k^+(D_x)f\|_{L^p(\mathbb{R}^2)},
\end{equation*}
for $k\geq1$ and $b\in\mathbb{R}$. For such purpose, we write
\begin{align}\label{eq320}
&\mathscr{F}^{-1}\left(\frac{(Q_b(a_{2m-1,\,k})-Q_b(\xi))\chi_k^+(\xi)\hat{f}(\tau,\xi)}{(\tau+P(\xi)+Q_b(\xi))(\tau+P(\xi)+iQ_b(a_{2m-1,\,k}))}\right)(t,x)\nonumber\\
=&\int_{\mathbb{R}}\int_{\mathbb{R}}\int_{\mathbb{R}}\frac{(2\pi)^{-1}e^{i((t-s)\tau+x\xi)}(Q_b(a_{2m-1,\,k})-Q_b(\xi))\chi_k^+(\xi)\tilde{f}(s,\xi)}{(\tau+P(\xi)+Q_b(\xi))(\tau+P(\xi)+iQ_b(a_{2m-1,\,k}))}dsd\xi d\tau\nonumber\\
=&\int_{\mathbb{R}}e^{it\tau}d\tau\int_\mathbb{R}e^{-is\tau}ds\int_{\mathbb{R}}\frac{(2\pi)^{-1}e^{ix\xi}(Q_b(a_{2m-1,\,k})-Q_b(\xi))\chi_k^+(\xi)\tilde{f}(s,\xi)}{(\tau+P(\xi)+Q_b(\xi))(\tau+P(\xi)+iQ_b(a_{2m-1,\,k}))}d\xi\nonumber\\
=&\int_{\mathbb{R}}e^{it\rho}d\rho\int_\mathbb{R}e^{-is\rho}ds\int_{\mathbb{R}}\frac{(2\pi)^{-1}e^{-i(t-s)P(\xi)}e^{ix\xi}(Q_b(a_{2m-1,\,k})-Q_b(\xi))\chi_k^+(\xi)\tilde{f}(s,\xi)}{(\rho+iQ_b(\xi))(\rho+iQ_b(a_{2m-1,\,k}))}d\xi\nonumber\\
=&\int_{\mathbb{R}}e^{it\rho}d\rho\int_\mathbb{R}e^{-is\rho}\left(e^{-i(t-s)P(D_x)}F_{b,\,k,\,\rho}(s,\cdot)\right)(x)ds,
\end{align}
where in the third equality we change the variable $\rho=\tau+P(\xi)$, and in the last equality $F_{b,\,k,\,\rho}(s,\cdot)$ is the inverse spatial Fourier transform of
\begin{equation}\label{eq321}
\tilde{F}_{b,\,k,\,\rho}(s,\xi)=\frac{(Q_b(a_{2m-1,\,k})-Q_b(\xi))\chi_k^+(\xi)\tilde{f}(s,\xi)}{(\rho+iQ_b(\xi))(\rho+iQ_b(a_{2m-1,\,k}))}.
\end{equation}
Notice that the decay estimate \eqref{eq314} and the Hardy-Littlewood-Sobolev inequality imply the following Strichartz type estimate
\begin{align*}\label{eq322}
\left\|\int_{\mathbb{R}}e^{-is\rho}\left(e^{-i(t-s)P(D_x)}F_{b,\,k,\,\rho}(s,\cdot)\right)(x)ds\right\|_{L_{t,x}^{p'}(\mathbb{R}^2)}\leq C\|F_{b,\,k,\,\rho}\|_{L^p(\mathbb{R}^2)},
\end{align*}
it then follows from Minkowski's inequality that
\begin{equation}\label{eq323}
\left\|\int_{\mathbb{R}}e^{it\rho}d\rho\int_\mathbb{R}e^{-is\rho}\left(e^{-i(t-s)P(D_x)}F_{b,\,k,\,\rho}(s,\cdot)\right)(x)ds\right\|_{L_{t,x}^{p'}(\mathbb{R}^2)}\leq C\int_{\mathbb{R}}\|F_{b,\,k,\,\rho}\|_{L^p(\mathbb{R}^2)}d\rho.
\end{equation}

Now we are left to study the $L^p-L^p$ bound associated with the Fourier multiplier in \eqref{eq321}. Since $\mathrm{supp}\,\chi_k^+\subset[a_{2m-1}+2^{k-2}|a_{2m-1}|,a_{2m-1}+2^{k-1}|a_{2m-1}|]$, if we take $\phi^+\in C_c^\infty((\frac12,\frac52))$ such that $0\leq\phi^+\leq1$ and $\phi^+\equiv1$ on $[1,2]$, denoted by
\begin{equation}\label{eq3.31}
\phi_k^+(\xi)=\phi^+\left(\frac{\xi-a_{2m-1}}{2^{k-2}|a_{2m-1}|}\right),
\end{equation}
we have $\phi_k^+\equiv1$ on $\mathrm{supp}\,\chi_k^+$. Then it suffices to show the $L^p-L^p$ bound associated with
\begin{equation*}
M_{b,\,k,\,\rho}^+(\xi)=\frac{(Q_b(a_{2m-1,\,k})-Q_b(\xi))\phi_k^+(\xi)}{(\rho+iQ_b(\xi))(\rho+iQ_b(a_{2m-1,\,k}))}.
\end{equation*}

When $\xi\in\mathrm{supp}\,M_{b,\,k,\,\rho}^+\subset[a_{2m-1}+2^{k-3}|a_{2m-1}|,a_{2m-1}+5\cdot2^{k-3}|a_{2m-1}|]$, recall \eqref{eq3.13}, we have for $1\leq j\leq2m-1$ and $k\geq1$ that
\begin{equation*}
\frac12\leq\frac{\xi-a_j}{a_{2m-1,\,k}-a_j}=\frac{\xi-a_j}{a_{2m-1}+2^{k-2}|a_{2m-1}|-a_j}\leq\frac52,
\end{equation*}
and consequently
\begin{equation}\label{eq3.34}
\frac12\leq\left|\frac{\xi-\xi_{b,\,j}}{a_{2m-1,\,k}-\xi_{b,\,j}}\right|=\sqrt{\frac{(\xi-a_j)^2+(\mathrm{Im}\,\xi_{b,\,j})^2}{(a_{2m-1,\,k}-a_j)^2+(\mathrm{Im}\,\xi_{b,\,j})^2}}\leq\frac52.
\end{equation}
Here we have used an elementary fact: if $x,\,y,\,z>0$, $C_1,\,C_2\geq1$, and $C_1^{-1}\leq\frac xy\leq C_2$, then $C_1^{-1}\leq\frac{x+z}{y+z}\leq C_2$. Now \eqref{eq3.34} implies
\begin{equation}\label{eq3.35}
C^{-1}|Q_b(a_{2m-1,\,k})|\leq|Q_b(\xi)|\leq C|Q_b(a_{2m-1,\,k})|,\quad\xi\in\mathrm{supp}\,M_{b,\,k,\,\rho}^+,
\end{equation}
where $C=(\frac52)^{2m-1}$. Also notice that when $\xi\in\mathrm{supp}\,M_{b,\,k,\,\rho}^+$ we have
\begin{equation*}
\frac{|\xi|}{a_{2m-1,\,k}-a_j}\leq\frac{|\xi|}{2^{k-2}|a_{2m-1}|}\leq\frac{|a_{2m-1}|+5\cdot2^{k-3}|a_{2m-1}|}{2^{k-2}|a_{2m-1}|}\leq\frac92,
\end{equation*}
which implies
\begin{equation}\label{eq3.37}
\begin{cases}
|\xi\partial_\xi Q_b(\xi)|\leq C|Q_b(a_{2m-1,\,k})|,\\
|\xi\partial_\xi\phi_k^{+}(\xi)|\leq C,
\end{cases}
\quad\xi\in\mathrm{supp}\,M_{b,\,k,\,\rho}^+.
\end{equation}
Thus by \eqref{eq3.35} and \eqref{eq3.37} we have
\begin{equation*}
|M_{b,\,k,\,\rho}^+(\xi)|+|\xi\partial_\xi M_{b,\,k,\,\rho}^+|\leq\frac{C|Q_b(a_{2m-1,\,k})|}{\rho^2+|Q_b(a_{2m-1,\,k})|^2}.
\end{equation*}
Then we apply the Mihlin multiplier theorem to obtain
\begin{equation}\label{eq3.39}
\|F_{b,\,k,\,\rho}\|_{L^p(\mathbb{R}^2)}\leq\frac{C|Q_b(a_{2m-1,\,k})|}{\rho^2+|Q_b(a_{2m-1,\,k})|^2}\|\chi_k^+(D_x)f\|_{L^p(\mathbb{R}^2)},
\end{equation}
and thus
\begin{equation}\label{eq327}
\begin{split}
\int_{\mathbb{R}}\|F_{b,\ k,\ \rho}\|_{L^p(\mathbb{R}^2)}d\rho\leq&C\|\chi_k^+(D_x)f\|_{L^p(\mathbb{R}^2)}\int_{\mathbb{R}}\frac{|Q_b(a_{2m-1,\ k})|}{\rho^2+|Q_b(a_{2m-1,\ k})|^2}d\rho\\
\leq&C\|\chi_k^+(D_x)f\|_{L^p(\mathbb{R}^2)}.
\end{split}
\end{equation}
In view of \eqref{eq320}, \eqref{eq323} and \eqref{eq327}, we have finished the proof of \eqref{eq315.b}.

\noindent\textbf{3. Proof of \eqref{eq315}.}\ 

The arguments are similar to those in Step 2. Let $\chi_{0,\,\nu}$ be the characteristic function of $(\frac{a_{\nu-1}+a_\nu}{2},\frac{a_\nu+a_{\nu+1}}{2}]$ for $1\leq\nu\leq2m-1$, where $a_0=a_1-|a_1|$ and $a_{2m}=a_{2m-1}+|a_{2m-1}|$, we have $\chi_0=\sum_{\nu=1}^{2m-1}\chi_{0,\,\nu}$. (If $a_\nu=a_{\nu-1}=a_{\nu+1}$, we define $\chi_{0,\nu}\equiv0$.) With $\chi^\pm$ defined at the beginning of Step 1, we also set
\begin{equation*}
\chi_{0,\,\nu,\,k}^+(\xi)=\chi^+\left(\frac{\xi-a_\nu}{2^{k-1}(a_{\nu+1}-a_\nu)}\right),\quad\chi_{0,\,\nu,\,k}^-(\xi)=\chi^-\left(\frac{\xi-a_\nu}{2^{k-1}(a_\nu-a_{\nu-1})}\right),\quad k\leq-1,
\end{equation*}
whenever $a_{\nu+1}>a_\nu$ or $a_\nu>a_{\nu-1}$. They are supported in $\mathrm{supp}\,\chi_{0,\,\nu}$, and
\begin{equation*}
\sum_{k=-\infty}^{-1}\left(\chi_{0,\,\nu,\,k}^+(\xi)+\chi_{0,\,\nu,\,k}^-(\xi)\right)=1,\quad\xi\in\mathrm{supp}\,\chi_{0,\,\nu}\setminus\{a_\nu\}.
\end{equation*}
With an argument similar to Step 1, we only have to focus on proving
\begin{equation}\label{eq3.43}
\left\|\mathscr{F}^{-1}(\chi_{0,\,\nu,\,k}^+(\xi)M_b(\tau,\xi)\hat{f}(\tau,\xi))\right\|_{L^{p'}(\mathbb{R}^2)}\leq C\|\chi_{0,\,\nu,\,k}^+(D_x)f\|_{L^p(\mathbb{R}^2)},
\end{equation}
where $\chi_{0,\,\nu,\,k}^+$ is non-trivial and $C>0$ is uniform in $\nu,\,k,\,b$.

Denoted by
\begin{equation*}
a_{0,\,\nu,\,k}=5\cdot 2^{k-2}(a_{\nu+1}-a_\nu),\quad Q_{b,\,\nu}(\xi)=Q_b(\xi+a_\nu)=2m\prod_{j=1}^{2m-1}(\xi+a_\nu-\xi_{b,\,j}),
\end{equation*}
and notice that $Q_{b,\,\nu}(a_{0,\,\nu,\,k})\neq0$, by Claim \ref{cl1.7} we have
\begin{equation*}
\left\|\mathscr{F}^{-1}\left(\frac{\chi_{0,\,\nu,\,k}^+(\xi)\hat{f}(\tau,\xi)}{\tau+P(\xi)+iQ_{b,\,\nu}(a_{0,\,\nu,\,k})}\right)\right\|_{L^{p'}(\mathbb{R}^2)}\leq C\|\chi_{0,\,\nu,\,k}^+(D_x)f\|_{L^p(\mathbb{R}^2)}.
\end{equation*}
By the same arguments in Step 2, it suffices to consider the $L^p-L^p$ bound associated with the Fourier multiplier
\begin{equation*}
M_{b,\,\nu,\,k,\,\rho}^+(\xi)=\frac{(Q_{b,\,\nu}(a_{0,\,\nu,\,k})-Q_b(\xi))\phi_{0,\,\nu,\,k}^+(\xi)}{(\rho+iQ_b(\xi))(\rho+iQ_{b,\,\nu}(a_{0,\,\nu,\,k}))},
\end{equation*}
where similar to \eqref{eq3.31} we are here using
\begin{equation*}
\phi_{0,\,\nu,\,k}^+(\xi)=\phi^+\left(\frac{\xi-a_\nu}{2^{k-1}(a_{\nu+1}-a_\nu)}\right).
\end{equation*}
Since $L^p-L^p$ bound is translation invariant for Fourier multiplier, we may instead consider $\tilde{M}_{b,\,\nu,\,k,\,\rho}^+(\xi)=M_{b,\,\nu,\,k,\,\rho}^+(\xi+a_\nu)$.

When $\xi\in\mathrm{supp}\,\tilde{M}_{b,\,\nu,\,k,\,\rho}^+\subset[2^{k-2}(a_{\nu+1}-a_\nu),5\cdot2^{k-2}(a_{\nu+1}-a_\nu)]$ and $k\leq-1$, one checks that
\begin{equation*}
\begin{cases}
1\leq\frac{\xi+a_\nu-a_j}{a_{0,\,\nu,\,k}+a_\nu-a_j}\leq\frac73,\quad \nu<j\leq2m-1,\\
\frac15\leq\frac{\xi+a_\nu-a_j}{a_{0,\,\nu,\,k}+a_\nu-a_j}\leq1,\quad 1\leq j\leq\nu,
\end{cases}
\end{equation*}
and
\begin{equation*}
\begin{cases}
\frac{|\xi|}{|a_{0,\,\nu,\,k}+a_\nu-a_j|}\leq\frac53,\quad1\leq j\leq2m-1,\\
\frac{|\xi|}{2^{k-1}(a_{\nu+1}-a_\nu)}\leq\frac52.
\end{cases}
\end{equation*}
A discussion parallel to \eqref{eq3.34}-\eqref{eq3.37} leads to
\begin{equation*}
|\tilde{M}_{b,\,\nu,\,k,\,\rho}^+(\xi)|+|\xi\partial_\xi\tilde{M}_{b,\,\nu,\,k,\,\rho}^+|\leq\frac{C|Q_{b,\,\nu}(a_{0,\,\nu,\,k})|}{\rho^2+|Q_{b,\,\nu}(a_{0,\,\nu,\,k})|^2}.
\end{equation*}
In the view of \eqref{eq3.39} and \eqref{eq327}, we have proved \eqref{eq3.43} and thus \eqref{eq315}. 

Now \eqref{eq39} is shown, and the proof of Lemma \ref{lem3.1} is complete.
\end{proof}

\noindent
\section*{Acknowledgments}
T. Huang is supported by the China Postdoctoral Science Foundation No. 2020M672929, and the Guangdong Basic and Applied Basic Research Foundation No. 2020A1515111048. S. Huang is supported by the National Natural Science Foundation of China No. 11801188 and No. 11971188. Q. Zheng is supported by the National Natural Science Foundation of China No. 11801188.

\end{document}